\documentclass[3p,11pt]{elsarticle}

\usepackage[english]{babel}
\usepackage[latin1]{inputenc} 
\usepackage[colorlinks,citecolor=red,urlcolor=red]{hyperref}
\usepackage[usenames,dvipsnames]{xcolor}
\usepackage{amsmath,amsthm,amssymb,amsfonts,graphicx,subfigure,moreverb}
\usepackage[ruled]{algorithm2e} 
\usepackage{myshortcuts} 
\usepackage{siunitx,tabularx,etoolbox,xcolor,booktabs} 
\usepackage{mathtools}
\usepackage{booktabs}

\newcommand{\fem}{\gamma} 
\newcommand{\mc}[1]{\mathcal{#1}} 
\newcommand{\mx}[1]{\displaystyle\mathbb{#1}} 
\newcommand{\dpar}[2]{\displaystyle\frac{\partial #1}{\partial #2}}
\newcommand{\Om}{\Omega}

\newcommand{\dis}{\displaystyle}

\newcommand{\C}{{\mathbb C}}
\newcommand{\cop}{{\mathcal F}} 

\newtheorem{thm}{Theorem}[section]
\newtheorem{prop}[thm]{Proposition} 
\newcommand{\numax}{{\nu_{\rm max}}}
\newcommand{\kmax}{{k_{\rm max}}}

\usepackage{pgfplots}
\usetikzlibrary{external}
\usetikzlibrary{decorations.pathreplacing}
\newcommand{\bftab}{\bfseries\small} 
\iftrue 
  \usepackage{tikz}
  \usepackage{pgfplots}
  \pgfplotsset{
    compat=newest,
    tick label style={font=\scriptsize},
    label style={font=\scriptsize},
    legend style={font=\scriptsize}
  }
  \usetikzlibrary{decorations.pathreplacing}
  \iffalse 
     \usepgfplotslibrary{external}
  \else
     \renewcommand{\tikzsetnextfilename}[1]{}
  \fi
\else
  \usepackage{tikzexternal}
  \tikzsetexternalprefix{gfx_tmp/}
  
\fi

\newtheorem{theorem}{Theorem}

\newtheorem{lemma}[theorem]{Lemma}

\newtheorem{remark}[theorem]{\textit{Remark}}

\journal{arXiv}

\begin{document}

\begin{frontmatter}

\title{Efficient resonance computations for Helmholtz problems based on a Dirichlet-to-Neumann map}

\author{Juan Carlos Araujo-Cabarcas\fnref{umit}}
\fntext[umit]{Department of Mathematics and Mathematical Statistics, Umeå University, MIT-Huset, 90187 Umeå, Sweden}
\author{Christian Engström\fnref{umit}}

\author{Elias Jarlebring\fnref{kth}}
\fntext[kth]{Department of Mathematics, Royal Institute of Technology (KTH), Stockholm, SeRC Swedish e-Science Research Center}

\begin{abstract}
We present an efficient procedure for computing resonances and resonant modes of Helmholtz problems posed in exterior domains. The problem is formulated as a nonlinear eigenvalue problem (NEP), where the nonlinearity arises from the use of a Dirichlet-to-Neumann map, which accounts for modeling unbounded domains. We consider a variational formulation and show that the spectrum consists of isolated eigenvalues of finite multiplicity that only can accumulate at infinity. The proposed method is based on a high order finite element discretization combined with a specialization of the Tensor Infinite Arnoldi method. Using Toeplitz matrices, we show how to specialize this method to our specific structure. In particular we introduce a pole cancellation technique in order to increase the radius of convergence for computation of eigenvalues that lie close to the poles of the matrix-valued function. The solution scheme can be applied to multiple resonators with a varying refractive index that is not necessarily piecewise constant. We present two test cases to show stability, performance and numerical accuracy of the method. In particular the use of a high order finite element discretization together with TIAR results in an efficient and reliable method to compute resonances.
\end{abstract}

\end{frontmatter}

\section{Introduction}\label{sec:intro} 

In this paper we consider the numerical approximation of resonances in an open system, where the solutions satisfy the Helmholtz equation for a given refractive index $\eta (x)$. In general, resonances of an operator are defined as poles of the resolvent operator taken in a particular generalized sense \cite{LectureZworski,MR1350074}. For Helmholtz equation Lenoir et al. \cite{lenoir92} have shown that resonances are solutions to a nonlinear eigenvalue problem (NEP) with a Dirichlet-to-Neumann (DtN) map $\mc G(\lambda )$ on an artificial boundary $\Gamma$. The pair $(u,\lambda)$ is a scattering resonance pair if
\begin{equation}
\begin{array}{rll}
\Delta u+\lambda ^2 \eta^2 u    = & 0 & \text{in}\,\,\Om, \\
\dpar{u}{n} = & \mc G(\lambda ) u & \text{on}\,\,\Gamma, \\
\end{array}
\label{eq:strong_DtN}
\end{equation}
where $\partial u/\partial n$ is the normal derivative and the non-negative function $\eta^2-1$ has compact support contained in the open domain $\Omega$. Hence, although our differential operator \eqref{eq:strong_DtN} is linear in $\lambda^2$ the DtN operator $\mc G(\lambda)$ (which we formalize in  Section~\ref{sect:background}) depends in a nonlinear way on the eigenvalue $\lambda$. 
In this work we present a new computational approach for approximating resonances of \eqref{eq:strong_DtN} 
in an efficient and accurate way. 
 
We present in Section~\ref{sec:variational} a variational formulation of the PDE \eqref{eq:strong_DtN} with the DtN map and show that the spectrum of the operator function consists of isolated eigenvalues of finite multiplicity, which accumulate only at infinity. This variational formulation is the base for the finite element method (FEM) in Section~\ref{sec:Disc}. In the FEM-implementation we discretize the operator function using high order Gauss-Lobatto shape functions and apply the $p$-version of the finite element method. The lower part of the spectrum of the operator function is then well approximated by the matrix function. Then very accurate approximations of the scattering resonance pairs are obtained if the nonlinear matrix eigenvalue problem can be solved accurately.

The considered NEP is of the type:
 find $\lambda\in\CC$ in an open subset of the complex plane and a non-zero $u\in\CC^n$ such that
\begin{equation}\label{eq:nep}
	T(\lambda)u=0.
\end{equation}
In our case $T$ is meromorphic in $\CC$, with poles in the region of interest defined as scaled roots of Hankel functions. Many 
numerical methods for the NEP \eqref{eq:nep}
have been developed in the numerical linear algebra community, in particular when $T$ is holomorphic in a large domain.
Since the NEP with an arbitrary normalization is a system of nonlinear equations,
Newton's method can be applied and considerably improved, e.g., by using block variations that can compute several eigenpairs simultaneously \cite{Kressner:2009:BLOCKNEWTON}. However, Newton-type methods bear the danger that some eigenvalues close to a given target could be missed.

There are also generalizations of successful
methods for linear eigenvalue problems, e.g., the nonlinear Arnoldi method \cite{Voss:2004:ARNOLDI},
the Jacobi-Davidson method \cite{Betcke:2004:JD} and LOBPCG \cite{Szyld:2014:LOBPCGTR}.
Numerical methods for  NEP can be based on contour integrals of the generalized resolvent $T^{-1}(\lambda)$ \cite{Asakura:2009:NUMERICAL,Beyn:2011:INTEGRAL,MR3423605,VanBarel:2016:CONTOUR,Ikegami:2010:FILTER}. The DtN-map is the source of the complicated nonlinearity in the eigenvalue parameter $\lambda$. 
Several other problems have been approached with artificial boundary conditions
and NEPs; see, e.g., the NEP arising in 
the modeling of an electromagnetic cavity \cite{Liao:2006:SOLVING}, 
the model of an optical fiber in \cite{Kaufman:2006:FIBER}
and bi-periodic slabs \cite{Tausch:2013:SLABS}. 
There are also several approaches leading to NEPs developed in the context of  photonic crystals 
\cite{Spence:2005:PHOTONIC,Fliss:2013:DIRICHLET,Klindworth:2014:DTN0,Dettmann:2009:INTERNAL}. To our knowledge, none of the methods developed in those papers have been adapted to resonance problems of our type.

Our methods belong to a class of methods which can be interpreted as Krylov methods, either for an infinite-dimensional problem, or as a dynamically increasing companion linearization of an approximation of the problem
\cite{Jarlebring:2012:INFARNOLDI,VanBeeumen:2013:RATIONAL,Guttel:2014:NLEIGS}.
For recent developments and problems see \cite{Voss:2013:NEPCHAPTER}  and \cite{Betcke:2013:NLEVPCOLL}.
Our approach is based on the infinite Arnoldi (IAR) method \cite{Jarlebring:2012:INFARNOLDI}.
In particular, we  adapt the variant with a tensor representation of the basis presented in \cite{Jarlebring:2015:WTIARTR}, called 
the tensor infinite Arnoldi method (TIAR). 
This method is designed to find all eigenvalues close to a given shift, where the radius of convergence depends on properties of the matrix-valued function \eqref{eq:nep}.
The infinite
Arnoldi method (both IAR and TIAR) requires access a procedure
to compute a quantity involving the derivatives of the problem. In Section~\ref{sec:TIAR} we derive an algorithm 
to compute the necessary 
quantities for our NEP which contains nonlinearities expressed as quotients of Hankel functions. 
In order to improve the convergence of the method, we introduce a pole cancellation technique 
which transforms the problem by removing poles.

In Section~\ref{sec:simulations} we provide a characterization of the performance of our approach.
In particular, we show that the new infinite Arnoldi method together with a $p$-FEM strategy is an efficient and reliable tool for resonance calculations with the DtN map.
 
\begin{figure}[!h]
	\centering
	\resizebox{5cm}{5cm}{%
		\begin{tikzpicture}[thick,scale=0.3, every node/.style={scale=0.8}]
			\tikzstyle{ann} = [fill=none,font=\large,inner sep=4pt]
		
			\draw( 0.00, 0.0) node { \includegraphics[scale=0.5]{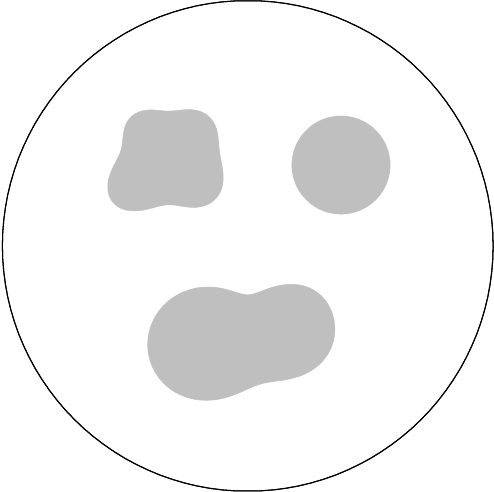} };
		
			\node[ann] at (2.8,0.45) {$a$};
			\draw[arrows=->,line width=0.8pt](0.0,0.0)--(5.55,0.0);
		
			\fill(0.0,0.0) circle (1pt);
		
			\node[ann] at ( 0.00, 4.0) {$\Omega_a$};
			\node[ann] at ( 4.80, 4.0) {$\Gamma_a$};
			\node[ann] at (-1.60, 1.9) {$\Omega_1$};
			\node[ann] at ( 2.20, 1.8) {$\Omega_2$};
			\node[ann] at ( 0.0,-2.00) {$\Omega_3$};
		
		\end{tikzpicture}
	}
	\caption{Example geometry of resonators $\Omega_i$, $i=1,2,3$ bounded by a circle.}
	\label{fig:domain}
\end{figure}
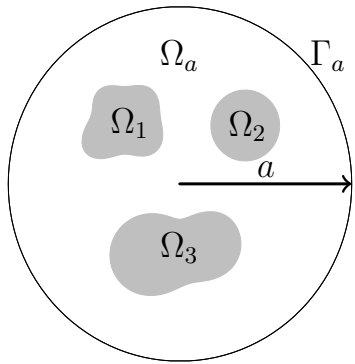

\section{Background and preliminaries}\label{sect:background}
Results for the problem \eqref{eq:strong_DtN} can be found in a considerable amount of literature; see \cite{MR996423,lenoir92,Colton+Kress1992} and references therein. For $\im \lambda >0$ we have uniqueness results \cite[Chapter VIII]{MR996423} and resonance values are therefore in the region $\im \lambda <0$. 

Let $\Omega_a\subset\mx R^2$ be an open disk of radius $a$ and boundary $\Gamma_a$. Assume $\eta\in L^{\infty}(\Om_a)$ and that the non-negative function $\eta^2-1$ has compact support contained in $\Omega_a$. 
A schematic setup of an example is illustrated in Figure~\ref{fig:domain}. The resonance problem restricted to $\Om_a$ is formally to find non-trivial solutions $(u,\lambda )$ such that \eqref{eq:strong_DtN} holds, where

the DtN operator on the circle $\Gamma_a$ has the explicit form
\begin{equation}
\mc G(\lambda )u:=\frac{1}{2\pi}
\sum_{\nu=-\infty}^{\infty}g_\nu(\lambda )\,e^{i\nu\theta}
\int_0^{2\pi}\!\!u(a,\theta')\,e^{-i\nu\theta'}d\theta',
\label{eq:DtN}
\end{equation}
where 
\begin{equation}  \label{eq:qdef}
g_\nu(\lambda ):=\lambda \frac{ H'_\nu(\lambda a)}{H_\nu(\lambda a)}
\end{equation}
and $\mc G(\lambda ):H^{1/2}(\Gamma_a)\rightarrow H^{-1/2}(\Gamma_a)$ is bounded \cite{lenoir92}. In the following, we identify the dual pairing $\langle\cdot,\cdot\rangle_{H^{-1/2}(\Gamma_a)\times H^{1/2}(\Gamma_a)}$ with the $L^2$-inner product $(\cdot,\cdot)_{\Gamma_a}$ over $\Gamma_a$. The theory presented in \cite{lenoir92} can with minor changes be used in the present case to derive properties of a variational formulation of the problem.
\subsection{Variational formulation}\label{sec:variational} 
Let $S$ denote the union of the set of zeros of $H_\nu(\lambda a),\,\nu\in \mx Z$. The eigenvalues of \eqref{eq:strong_DtN} are determined by the following variational problem:
Find $u\in H^1(\Om_a)\setminus \{0\}$ and $\lambda \in\mc D:=\mx C\setminus \{\mx R^-\cup S\}$ such that for all $v\in H^1(\Om_a)$
\begin{equation}
	(\nabla u,\nabla v)_{\Om_a}-\lambda^2(\eta^2 u,v)_{\Om_a}-(\mc G(\lambda ) u,v)_{\Gamma_a}=0,
	\label{eq:eig_prob}
\end{equation}
where $(u,v)_{\Om_a}:=\int_{\Om_a}u\bar v\,dx$, $(\nabla u,\nabla v)_{\Om_a}:=\int_{\Om_a}\nabla u\cdot\nabla\bar v\,dx$, and 
\begin{equation}
(\mc Gu,v)_{\Gamma_a}=\dis \sum_{\nu=-\infty}^{\infty}\lambda a\frac{H'_\nu(\lambda a)}{H_\nu(\lambda a)}\,\hat u_\nu\, \bar {\hat v}_\nu,\quad \hat \varphi_\nu=\frac{1}{\sqrt{2\pi}}\int_{0}^{2\pi}\varphi(a,\theta')\,e^{-i\nu\theta'}d\theta'.
\label{eq:dtn_1}
\end{equation}
Let $(u,v)_1:=(\nabla u,\nabla v)_{\Om_a}+(\eta^2 u,v)_{\Om_a}$. Then the norm $\norm{u}_1:=\sqrt{(u,u)_1}$ is equivalent to the standard norm on $H^1(\Om_a)$. Define the operator
$\cop(\lambda ):H^1(\Om_a)\rightarrow H^1(\Om_a)$ by
\begin{equation}
	(\cop(\lambda ) u,v)_1:=(\lambda^2+1)(\eta^2 u,v)_{\Om_a}+(\mc G(\lambda ) u,v)_{\Gamma_a}.
	\label{eq:op_T}
\end{equation}
An operator formulation of \eqref{eq:eig_prob} is to find $u\in H^1(\Om_a)\setminus \{0\}$ and $\lambda \in\mc D$ such that 
\begin{equation}
	(I-\cop(\lambda )) u = 0.
	\label{eq:equiv_eig}
\end{equation}
Let $\mc L(H^1(\Om_a))$ denote the space of bounded linear operators on $H^1(\Om_a)$. An operator $A\in\mc L(H^1(\Om_a))$ is Fredholm if it has finite-dimensional kernel $\text{ker}\,A$ and cokernal $\text{coker}\,A$. The index of a Fredholm operator is $\text{ind}\,A:=\text{dim}\,\text{ker}\,A-\text{dim}\,\text{coker}\,A$. The essential spectrum of $A$ denoted $\sigma_{\rm{ess}}(A)$ is the set of of complex numbers such that $A-\lambda I$ is not Fredholm. Let $\sigma_{\infty}(A)\subset\sigma_{\rm{ess}}(A)$ denote the set of eigenvalues of infinite multiplicity.

Let $\mc G_\numax(\lambda )$ denote the operator \eqref{eq:DtN} truncated after $|\nu|=\numax$ and let $I-\cop_\numax(\lambda )$, $\lambda \in\mc D$ denote the operator defined by 
\begin{equation}
	((I-\cop_\numax(\lambda ))u,v)_1:=(u,v)_1-(\lambda^2+1)(\eta^2 u,v)_{\Om_a}-(\mc G_\numax(\lambda ) u,v)_{\Gamma_a}.
	\label{eq:op_Tl}
\end{equation}
The operator function with a truncated DtN map \eqref{eq:op_Tl} is the base for the numerical method and the following proposition shows that the spectrum is discrete with no finite accumulation point in the right-half plane.
\begin{prop}\label{prop:T}
Let $\numax\in\{0,1,\dots,\infty\}$. Then the $H^1(\Om_a)$-spectrum of the operator valued function  $I-\cop_\numax:\mc D\rightarrow \mc L(H^1(\Om_a))$ consists of isolated eigenvalues of finite multiplicity. Moreover, these eigenvalues can not have a finite accumulation point in the right-half plane.
\end{prop}
\begin{proof}
The function $\cop_\numax:\mc D\rightarrow \mc L(H^1(\Om_a))$ is holomorphic \cite[prop 4]{lenoir92} and $I-\cop_\numax(\lambda )$ is a Fredholm operator of index zero \cite[prop 5]{lenoir92}. Moreover, we have
\begin{equation}\label{eq:MBessel}
	ia\frac{ H'_\nu(ia)}{H_\nu(ia)}=a\frac{ K'_\nu(a)}{K_\nu(a)},
\end{equation}
where $K_\nu(a)=\frac{\pi}{2}i^{\nu+1}H_\nu (ia)$, $\nu\in\mx Z$ are the modified Bessel functions \cite{MR1349110}. The expression \eqref{eq:MBessel} is negative since $K_\nu(a)>0$ and $K'_\nu(a)<0$ for $a>0$ \cite[p. 181]{MR1349110}. Hence $(\mc G_\numax(i) u,u)_{\Gamma_a}<0$ and
\[
	(u,u)_1-(\cop_\numax(i) u,u)_1=(u,u)_1-(\mc G_\numax(i) u,u)_{\Gamma_a}\geq \|u\|_1^2,
\]
which shows that the resolvent set of $I-\cop$ is non-empty. Hence, from the analytical Fredholm theorem follows that the spectrum is discrete and all eigenvalues are of finite multiplicity \cite[Theorem 5.1]{Gohberg+Krein1969}. In the right half-plane the eigenvalues can therefore only accumulate at the poles. The location of the poles is a function of $a$, but the spectrum is for large enough $a$ independent of $a$, which implies that we have no accumulation in the right half-plane.
\end{proof}
In Section \ref{sec:pole_cancel}, we improve convergence of the infinite Arnoldi method by multiplying the original matrix-valued function with a polynomial. In Proposition \ref{prop:tildeT}, we prove basic properties of the underlaying operator-valued function. This proposition shows that the canceled poles will be in the essential spectrum of the modified operator.

\begin{prop}\label{prop:tildeT}
Take $\mc B\subset\overline{\mc D}$ in the right-half plane such that $\cop$ has exactly one pole $z$ in $\mc B$. Then is $\tilde{T}(\lambda ):H^1(\Om_a)\rightarrow H^1(\Om_a)$,
\begin{equation}
	\tilde{\mc T}(\lambda):=(\lambda-z)I-\tilde{\mc F}(\lambda),\quad \tilde{\mc F}(\lambda):=(\lambda-z)\mc F (\lambda)
\end{equation}
holomorphic in $\mc B$, $\{z\}=\sigma_{\rm{ess}}(\tilde{\mc{T}})=\sigma_{\infty}(\tilde{\mc{T}})$, and $\tilde{\mc{T}}(z)$ is of rank one.
\end{prop}
\begin{proof}
Take without loss of generality $a=1$. Since all zeros of $H_\nu(\lambda)$ are simple \cite[p 370]{Abramowitz+Stegun1970}, we have by definition $H_\nu(\lambda)=(\lambda-z)f(\lambda)$ with $f$ holomorphic and $f(z)\neq 0$. Hence,
\[
	\frac{H'_\nu(\lambda)}{H_\nu(\lambda)}=\frac{1}{\lambda-z}+\frac{f'(\lambda)}{f(\lambda)},
\]
which implies that $\tilde{\mc T}$ is holomorphic in $\mc B$. Since $H^1(\Om_a)$ is infinite dimensional, it follows that $\tilde{\mc{T}}(z)$ is not Fredholm.  Assume $H_{\nu_z}(z)=0$,  $H_{\nu}(z)\neq 0$, $\nu\neq\nu_z$. Then  $(\tilde{\mc{T}}(z)u,v)_1=-z\hat u_{\nu_z}\, \bar {\hat v}_{\nu_z}$, which implies that $\tilde{\mc{T}}(z)$ and  $\tilde{\mc{T} }(z)^*$ are of
rank one and hence $z\in\sigma_{\infty}(\tilde{\mc{T}})$. 
\end{proof}

\begin{remark}
Proposition \ref{prop:tildeT} shows that a pole $z$ after multiplication by $\lambda-z$ is an eigenvalue of infinite multiplicity of the new operator function. This change of the spectral properties also apply more generally. Let  $H$ denote an infinite dimensional Hilbert space and assume that $F$ is a $\mc L(H)$-valued finitely meromorphic function at $z$. Hence,
\[
	F(\lambda)=\sum_{m=-s}^{\infty}(\lambda-z)^m F_m,
\]
where $F_{-s},F_{-s+1},\dots,F_{-1}$ are finite rank. Define the bounded operator function 
\begin{equation}
	\tilde{\mc T}(\lambda):=(\lambda-z)^sI-\tilde{\mc F}(\lambda),\quad \tilde{\mc F}(\lambda):=(\lambda-z)^s\mc F (\lambda).
\end{equation}
Then  $z\in\sigma_{\infty}(\tilde{\mc{T}})$, since $\tilde{\mc T}(z)=-F_{-s}$ is finite rank. However if $F_{-s}$ is not finite rank, $z$ will in general not be an eigenvalue of infinite multiplicity; see for example \cite{MR3423605}.
\end{remark}

\section{Discretization with the finite element method}\label{sec:Disc}
The disk $\Omega_a$, depicted in Figure \ref{fig:domain}, is covered with a regular and quasi uniform finite element mesh $\mc T$ consisting of quadrilateral elements $\{K_i\}^N_{i=1}$. Let $\rho_i$ be the diameter of the largest ball contained in $K_i$ and denote by $h$ the maximum mesh size $h:=\max{\rho_i}$. Let $\mc P_p$ denote the space of polynomials on $\mx R^2$ of degree $\leq p$ and set $\fem:=\{h,p\}$ \cite[Ch 4]{Schwab1998}. We define the finite element space $S^{\fem}(\Omega_a):=\{u\in H^1(\Omega_a):\left. u\right|_{K_i} \in\mc P_p(K_i)\,\,\hbox{for}\,\,K_i\in\mc T\}$, and $N:=\dim(S^{\fem}(\Omega_a))$ \cite{Babuska92}. Furthermore, all our computations are done in the approximated domain $\Omega_a^\fem$ by using curvilinear elements following standard procedures \cite{Babuska92}.

From \eqref{eq:eig_prob} we state the finite element problem: Find $u_\fem\in S^{\fem}(\Om_a^\fem)\setminus \{0\}$ and $\lambda_\fem \in\mc D$ such that
\begin{equation}
	(\nabla u_\fem,\nabla v)_{\Om^\fem_a}-\lambda_\fem^2(\eta^2 u_\fem,v)_{\Om^\fem_a}-(\mc{G}_\numax(\lambda_\fem ) u_\fem,v)_{\Gamma^\fem_a}=0\,\, \hbox{for all}\,\, v\in S^{\fem}(\Om_a^\fem).
	\label{eq:eig_fem}
\end{equation}
We showed in Section \ref{sec:variational} that the analytic operator function $I-\cop_\numax$ with a truncated DtN map is Fredholm-valued on $\mc D$ and that the resolvent set is non-empty. Hence, all eigenvalues are isolated and of finite multiplicity. Moreover, from Karma \cite{Karma1996a,Karma1996b} follows that any sequence $\{\lambda_\fem\}$, dim $S^{\fem}(\Om_a^\fem)\rightarrow\infty$ of approximative eigenvalues of \eqref{eq:eig_fem} converges to an eigenvalue $\lambda$ of \eqref{eq:eig_prob}. The eigenfunctions are in $H^2(\Omega_a)$ and are piecewise analytic if the interfaces are analytic curves. Optimal convergence is expected under the assumptions that all interfaces are resolved by curvilinear cells and the eigenvalues are semi-simple. Hence, exponential convergence is expected with $p$-FEM and optimal converge rates are expected with $h$-FEM \cite{Babuska+Buo+Osborn1989,Hae+Babuska}.

\subsection{Assembly of the FE matrices}\label{sec:DtN_2d}
Let $\{\varphi_1,\dots,\varphi_N\}$ be a basis of $S^{\fem}(\Omega_a^\fem)$. Then $u_\fem\in S^{\fem}(\Omega_a^\fem)$ can be written in the form
\begin{equation}
	u_\fem=\sum_{j=1}^N \xi_j\,\varphi_j
\label{eq:fem_ans}
\end{equation}
and the entries in the finite element matrices are
\begin{equation}
A_{ij}= (\nabla\varphi_j, \nabla\varphi_i)_{\Om^\fem_a},\quad M_{ij}=(\eta^2\,\varphi_j, \varphi_i)_{\Om^\fem_a},\quad G_{ij}(\lambda)=\sum_{\nu=-l}^{l}\lambda a\frac{H'_\nu(\lambda a)}{H_\nu(\lambda a)}\,Q^\nu_{ij}
\label{eq:def_fem}
\end{equation}
with 
 \begin{equation}
	Q^\nu_{ij}=\hat{\varphi}_j^\nu\, \bar{ \hat{\varphi}}_i^\nu,\quad \hat \varphi_j^\nu=\frac{1}{\sqrt{2\pi}}\int_{0}^{2\pi}\varphi_j(a,\theta')\,e^{-i\nu\theta'}d\theta', \quad i,j=1,\dots,N.
\label{eq:dtn_1}
\end{equation}
The nonlinear matrix eigenvalue problem is then: Find the eigenpars $(\lambda ,{\xi})\in\mc D\times \C^N\setminus \{0\}$ such that
\begin{equation}
T (\lambda ) \,{\xi}:=\left(A-\lambda^2M-G(\lambda)\right){\xi}=0.
\label{eq:eig_DtN}
\end{equation}
In the assembly routine, we only store the vectors $q_j^\nu=\hat{\varphi}_j^\nu$ and then compute the matrices as $Q^\nu = q^\nu\otimes \bar q^\nu$. Notice that $q_j^\nu=0$ for nodes $x_j$ such that $\hbox{supp}(\varphi_j)\cap \Gamma_a=\emptyset$. If the ordering of nodes is such that the $N_a$ boundary nodes are placed first, then the $Q^\nu$ matrices have a dense upper-left block of size $N_a\times N_a$. This lost of sparsity is taken into account in the assembly routine, where we allow extra entries in the sparsitty pattern of the FE matrices.

It is a standard technique to use Gauss-Legendre quadratures for the evaluation of integrals in FE. However, the trace integral \eqref{eq:dtn_1} requires that the number of quadrature points is increased linearly with $\nu$, because the integrand contains $e^{i\nu \theta(x_1,\,x_2)}$.

All numerical experiments have been carried out using the finite element library \emph{deal.II} \cite{dealII82} with Gauss-Lobatto shape functions \cite[Sec. 1.2.3]{solin04}. For fast assembly and computations with complex numbers the package PETSc \cite{petsc-efficient} is used.

The computational platform was provided by the High-Performance Computing Center North (HPC2N) at Ume\aa\, University, and all experiments were run on the distributed memory system Abisko. The jobs were run in serial on an exclusive node: during the process, no other jobs were running on the same node. Node specifications: four AMD Opteron 6238 processors with a total of 48 cores per node.

\section{Specialization of the infinite Arnoldi method}\label{sec:TIAR}

Our algorithm for solving the NEP \eqref{eq:eig_DtN} is based on the tensor infinite Arnoldi method  (TIAR) \cite[Algorithm~2]{Jarlebring:2015:WTIARTR}, which is an improvement of the infinite Arnoldi method (IAR) \cite{Jarlebring:2012:INFARNOLDI}.
The version of the infinite Arnoldi method considered here can be viewed as the standard Arnoldi method applied to the companion linearization arising from a Taylor expansion of an analytic matrix-valued function  $T$ \cite{Jarlebring:2012:INFARNOLDI}.
More precisely, it can be derived from a particular companion linearization of the Taylor expansion of $T$ and Arnoldi's method
for eigenvalue problems. The particular choice of companion linearization
provides a structure such that the truncation parameter in a certain sense can be increased to infinity and the algorithm is therefore equivalent to
Arnoldi's method applied to an infinite matrix.

Hence, the algorithm generates a Hessenberg matrix and the eigenvalues of this matrix correspond to
eigenvalue approximations of the NEP.

TIAR is a variant of IAR where a tensor representation of the basis of IAR is used, which results in an algorithm of the same complexity as IAR but it requires less memory.
 Moreover, the usage of tensors in
  TIAR makes it considerably more efficient than IAR for certain types of problems
  \cite{Jarlebring:2015:WTIARTR} on modern computer architectures. We derive below a new version of TIAR adapted to the special structure of \eqref{eq:eig_DtN}.

\subsection{Quantities required for the infinite Arnoldi method}
All of the variants of the infinite Arnoldi method
require (in some way) access to derivatives. In TIAR we need a procedure to compute
\begin{equation}\label{eq:x0}
  x_0=-T(\mu)^{-1}\left(\sum_{i=1}^{k}T^{(i)}(\mu)x_i\right),
\end{equation}
where $k$ is the iteration count, and we compute eigenvalues close to a target $\mu$.

The matrix $T(\mu)$ is independent of $k$ and we compute an LU-factorization of $T(\mu)$ before starting the iterations in TIAR. Without loss of generality, we write $T$ of the form 
\begin{equation}\label{eq:SMF}
  T(\lambda)=\sum_{j=1}^NA_j f_j(\lambda).
\end{equation}
Then, we express the NEP \eqref{eq:eig_DtN} in the form \eqref{eq:SMF} with
\begin{subequations}\label{eq:first_terms}
\begin{eqnarray}
  A_1&=&A,\;\;\; f_1(\lambda)=1   \label{eq:first_terms_A1}\\
  A_2&=&M,\;\;\; f_2(\lambda)=-\lambda^2 \label{eq:first_terms_A2}
\end{eqnarray}
\end{subequations}
and for $j=3,\ldots,2\numax+3$, 
\begin{subequations}\label{eq:second_terms}
\begin{eqnarray}
  A_j&=&-a q_{j-\numax-3}q_{j-\numax-3}^T   \\
  f_j(\lambda)&=&g_{j-\numax-3}(\lambda)=\frac{H_{j-\numax-3}'(a\lambda)}{H_{j-\numax-3}(a\lambda)}\lambda . \label{eq:second_terms_f}
\end{eqnarray}
\end{subequations}
Hence,  \eqref{eq:x0} for the problem stated in \eqref{eq:eig_DtN} can be written as
\begin{equation}\label{eq:x0_SMF}
  x_0=-T(\mu)^{-1}\left(\sum_{j=1}^{2\numax+3} A_j\sum_{i=1}^{k} x_if_j^{(i)}(\mu)\right).
\end{equation}
In order to specialize TIAR to the considered NEP a procedure to compute the derivatives in \eqref{eq:x0_SMF} is required.

\subsection{High-order quotient and products rules with Toeplitz matrices}\label{sec:toep}
The nonlinearities \eqref{eq:second_terms_f}  are products and quotients
of analytic functions of the form
\begin{equation}  \label{eq:g_general_form}
     g(\lambda)=\frac{h_1(\lambda)}{h_2(\lambda)}p(\lambda),
\end{equation}
where $h_1$ and $h_2$ are analytic functions and $p$ a polynomial.
Formula \eqref{eq:x0_SMF} includes (high-order) derivatives of such scalar functions.
There are high-order product and quotient rules for
differentiation, which are explicit and
lead to formulas for the high-order derivatives of \eqref{eq:g_general_form}. 
However, the direct application of the
high-order product and quotient rules, i.e., the general Leibniz rule,
are somewhat unsatisfactory in terms of computation
time.
Here we propose to use a formula involving 
Toeplitz matrices, and compute directly a given number of derivatives 
using only matrix-vector operations.

The derivation of our Toeplitz matrix computational formula for the
derivatives, can be done by comparing terms in a Taylor expansion of $g(\lambda)$.
For our purposes it is more natural to 
use matrix functions in the sense
of \cite{Higham:2008:MATFUN}. 
Let $J_\mu$ denote a Jordan matrix with eigenvalue $\mu$,
\begin{equation}\label{eq:jordan}
J_{\mu}=
\small \begin{bmatrix}
  \mu&1      &        &    \\
     &\ddots &\ddots  &    \\
     &       &\ddots  & 1\\
     &       &        &\mu
\end{bmatrix}.
\end{equation}
The matrix function of a Jordan matrix is a triangular Toeplitz matrix \cite[Definition~1.2]{Higham:2008:MATFUN} containing scaled derivatives of the scalar function,
\begin{equation}\label{eq:f_jordan}
  f(J_\mu^T)=
f(J_\mu)^T=
  \begin{bmatrix}
    \frac{f(\mu)}{0!}&        & \\
    \vdots& \ddots & \\
    \frac{f^{(p-1)}(\mu)}{(p-1)!}&\cdots    & \frac{f(\mu)}{0!}
  \end{bmatrix}.
\end{equation}
Therefore, the derivatives of a function for which the corresponding 
matrix function is available can be computed with the formula
\begin{equation}\label{eq:matfun_ders}
\begin{bmatrix}
f(\mu)\\
f'(\mu)\\
\vdots\\
f^{(p-1)}(\mu)
\end{bmatrix}=\operatorname{diag}(u)f(J_\mu^T)e_1,
\end{equation}
where $u^T=[0!,\ldots,(p-1)!]$. Then, using formula \eqref{eq:matfun_ders}, we find that
the derivatives \eqref{eq:g_general_form}
are given by 
\begin{equation}\label{eq:g_ders}
\begin{bmatrix}
g(\mu)\\
g'(\mu)\\
\vdots\\
g^{(p-1)}(\mu)
\end{bmatrix}=\operatorname{diag}(u)h_1(J_\mu^T)h_2(J_\mu^T)^{-1}p(J_\mu^T)e_1.
\end{equation}
The TIAR algorithm can be applied to any NEP expressed in the form \eqref{eq:SMF} for which efficient and reliable computation of the corresponding matrix function is available.
Such matrix function representation 
is analogous to the input for several other methods and software
packages,
e.g., the block Newton method \cite{Kressner:2009:BLOCKNEWTON} 
and NLEIGS \cite{Guttel:2014:NLEIGS}. 
This representation is suitable for
 many NEPs including the problem considered in this paper
 (due to the derivative computation specialization provided here),
  but not in general for NEPs where the number of terms is large.
Completely analogous to  scalar-valued functions, 
products and quotients of matrix functions
are expressed as matrix multiplies and inverses 
(which is formally a consequence of \cite[Theorem~1.15]{Higham:2008:MATFUN}).
\subsection{Derivative computations of Hankel functions}

For the purpose of using formula \eqref{eq:g_ders}, we require the matrix functions corresponding
to \eqref{eq:first_terms}  and \eqref{eq:second_terms_f}.
These matrix functions
are $f_1(S)=I$, $f_2(S)=-S^2$, and 
\begin{equation}\label{eq:f_qoutient}
  f_j(S)=g_{j-\numax-3}(S)=H_{j-\numax-3}'(aS)H_{j-\numax-3}(aS)^{-1}S, \;\;j=3,\ldots,2\numax+3.
\end{equation}
The matrix function \eqref{eq:f_qoutient} could be computed
directly if robust and efficient
methods for matrix Hankel functions were available. 
To the best of our knowledge, there are unfortunately no such specialized methods.

Rather than computing the matrix function of the
Hankel function explicitly, we use that for applying \eqref{eq:g_ders} only
the matrix function of a transposed Jordan block is required. Hence, we compute the Toeplitz matrix \eqref{eq:f_jordan} 
consisting of scaled derivatives.
Then, the derivatives of the Hankel function 
can be robustly computed as follows.
For notational convenience, let  for any $r\in\NN$,
\begin{subequations}\label{eq:hankel_vec}
\begin{eqnarray}
  \underline{H}_r(z)&:=&[H_0(z),H_1(z),\ldots,H_{r-1}(z)]^T  \\
  \underline{H}_{-r}(z)&:=&[H_0(z),H_{-1}(z),\ldots,H_{1-r}(z)]^T  \\
  &=&D_r\underline{H}_r(z),\label{eq:hankel_vec_neg}
\end{eqnarray}
\end{subequations}
where $D_r=\diag(1,-1,1,-1,\ldots)\in\RR^r$
and \eqref{eq:hankel_vec_neg} follows from the symmetry 
of Hankel functions.
Using the recursion formulas for Hankel functions, the infinite
vector of derivatives of the Hankel functions
can be written as 
\begin{equation}\label{eq:hankel_der_recursion}
   \underline{H}'_\infty(z)=B_\infty \underline{H}_\infty(z)
\end{equation}
where $B_\infty$ is given by the infinite extension of
a matrix $B_r$ formed by the sum of a Toeplitz matrix and a rank-one matrix,
\begin{equation}\label{eq:Br}
   B_r=
   \begin{bmatrix}
     0& -1     &       & \\
  1/2 & \ddots &-1/2 &\\
      & \ddots & \ddots &\ddots &\\
      & &  1/2& \ddots & -1/2\\
      & &  & 1/2       & 0
   \end{bmatrix}\in\RR^{r\times r}.
\end{equation}
The relation \eqref{eq:hankel_der_recursion} leads
to a procedure to compute the derivatives 
summarized in the following lemma.  
\begin{lemma}[Hankel function derivative recursions]\label{lem:hankelder}
Let the vector of Hankel functions be defined by \eqref{eq:hankel_vec}
and the tridiagonal matrix $B_r$ in \eqref{eq:Br}. Then,
the $k$th derivatives of Hankel functions are given by
\begin{equation}\label{eq:hankel_der_vec}
   \frac{d^k}{dz^k}\underline{H}_r(az)=a^k[I_r,0,\ldots,0]B_{r+k}^k\underline{H}_{r+k}(az)
\end{equation}
and $\underline{H}^{(k)}_{-r}(z)=D_r\underline{H}^{(k)}_r(z)$.
\end{lemma}
\begin{proof}

By the chain rule and repeated application of \eqref{eq:hankel_der_recursion}
we have for an arbitrary derivative $k$,
\[
\frac{d^k}{dz^k}\underline{H}_\infty(az)= a^k   \underline{H}^{(k)}_\infty(az) 
   =a^kB_\infty^k \underline{H}_\infty(az).  
\]
The matrix $B_\infty$ is tridiagonal. Therefore, 

\begin{equation}\label{eq:hankel_der_proof_1}
[I_r,0,\ldots] B_\infty^k= [I_r,0,\ldots,0][B_{k+r}^k,0,\ldots,0].
\end{equation}
Then \eqref{eq:hankel_der_vec}
follows from \eqref{eq:hankel_der_proof_1}
with 
\[
\underline{H}^{(k)}_r(z)=[I_r,0,\ldots ,0] \underline{H}^{(k)}_\infty(z)=[I_r,0,\ldots ,0] B_\infty^k \underline{H}_\infty(z).
\]
\end{proof}
\begin{remark}[Alternative ways to compute derivatives of Hankel functions]
There are in principle, many ways to numerically 
compute derivatives of Hankel functions, e.g.,
with various discretization schemes.
An advantage of our approach is that it is exact in exact arithmetic, 
and appears relatively insensitive to round-off errors. 
We describe in Algorithm 1 how the computation can 
be performed with only matrix-vector products. 
This is easily integrated with the pole cancellation technique
described in Section~\ref{sec:pole_cancel}, and the
large number of derivatives required in TIAR.
\end{remark}

\subsection{Pole cancellation and derivative computation algorithm}\label{sec:pole_cancel}

Lemma~\ref{lem:hankelder} does provide
a procedure to compute derivatives of the Hankel functions
and it can be directly used to  specialize TIAR to \eqref{eq:eig_DtN}.
We show in Section~\ref{sec:simulations}
that the direct application of the method numerically works well for some regions
of the complex plane, but unfortunately 
not for the entire complex plane.
The infinite Arnoldi method
is designed for problems which are analytic
in a large domain, and convergence cannot 
be guaranteed for eigenvalues outside the convergence
disk for the power series expansion at $\mu$. Below, use a transformation of the matrix function
to increase the convergence radius by effectively cancelling poles. A similar holomorphic extension was successfully applied to a matrix-valued function that is used to determine surface waves in soil mechanics \cite[Section~7.3.1]{Beeumen2015}. Suppose $z_i$, $i=1,\ldots,p$ are zeros of Hankel functions and define
\[
  \tilde{T}(\lambda):=(\lambda-z_1)\cdots(\lambda-z_p)T(\lambda).
\]
From Proposition \ref{prop:tildeT} follows that $\tilde{T}$ is holomorphic in a domain containing $z_i$, $i=1,\ldots,p$. Hence, we have a decomposition analogous to 
\eqref{eq:SMF} where $\tilde{T}(\lambda)=\sum_{i}^NA_i \tilde{f}_i(\lambda)$,
with
\[
\tilde{f}_i(\lambda):=(\lambda-z_1)\cdots(\lambda-z_p)f_i(\lambda).
\]
Note that the nonlinear terms of $\tilde T$ are
 $\tilde{g}_m(\lambda)=(\lambda-z_1)\cdots(\lambda-z_p)g_m(\lambda)$, $m=-\numax,\ldots,m$, which are terms of the form \eqref{eq:g_general_form} with $p(\lambda):=(\lambda_1-z_1)\cdots(\lambda_1-z_p)\lambda$. 
Therefore, we can still use formula \eqref{eq:g_ders}.
By construction, the nonlinear (matrix) eigenvalue problem associated
with $\tilde{T}$ has
the same solutions as the NEP associated with $T$, 
except for possibly $\lambda=z_i$. 
Moreover, the singularity set of $\tilde{T}$ is the same as $T$
except that $z_1,\ldots,z_p$ are not poles of $\tilde{T}$. 
The values $\lambda=z_i$ 
are not solutions to the NEP $T$, since $z_i$ is
a pole. However, Proposition \ref{prop:tildeT} shows that the poles of the original operator function are eigenvalues of infinite multiplicity of the operator function obtained by multiplying with a polynomial. 
The matrix function $\tilde{T}$ will then numerically have several additional eigenvalues close to the poles compared to the eigenvalues of $T$. It is therefore essential that an estimate of the quality of a computed eigenpair of $\tilde{T}$ is based on the original matrix function $T$. In Section~\ref{sec:simulations}, we estimate the quality of a computed eigenpar by the standard backward error estimate 
(consistent e.g. with \cite{Liao:2010:NLRR}) for NEPs
\[
  \|T(\lambda)v\|/\alpha(\lambda,v),\,\,\hbox{with}\,\,
  \alpha(\lambda,v):=\|A\|+|\lambda|^2\|M\|+
  \sum_{\nu=-l}^l|\lambda| a|H'_\nu(a\lambda)|/|H_\nu(a\lambda)|\|Q^{\nu}\|.
\]
We summarize the combination of
\eqref{eq:g_ders} with the Hankel function
derivative computation in Lemma~\ref{lem:hankelder} in Algorithm~\ref{alg:dercomp}.
In the algorithm we have taken advantage of the structure of
the numerator and denominator in \eqref{eq:g_general_form},
and that the derivatives corresponding to all needed
indexes can be computed simultaneously.
The output of the algorithm is the matrix consisting of derivatives 
of $\tilde{g}_{i}$, 
\begin{equation}  \label{eq:dermat}
X=
\begin{bmatrix}
x_1 & \cdots & x_{\numax+1}
\end{bmatrix}=
\begin{bmatrix}
  \tilde{g}_0^{(0)}& \cdots  & \tilde{g}_\numax^{(0)}\\
   \vdots          &         &  \vdots\\
  \tilde{g}_0^{(\kmax)}& \cdots     &  \tilde{g}_\numax^{(\kmax)}\ \\
\end{bmatrix}.
\end{equation}
Note that $\tilde{g}_i=\tilde{g}_{-i}$ such that we can use \eqref{eq:dermat}
also for negative index.

\begin{algorithm} 
\caption{Derivative computation for \eqref{eq:dermat} \label{alg:dercomp}}
\SetKwInOut{Input}{input}\SetKwInOut{Output}{output}
\Input{Number of derivatives $\kmax$, largest index in modulus $\numax$,
sequence of poles $z_1,\ldots,z_p$}
\Output{The matrix $X$ in \eqref{eq:dermat} consisting of derivatives of $\tilde{g}_i$} 
\BlankLine 
\nl Compute $r_0=\underline{H}_{\kmax+\numax+1}(\mu)\in\CC^{\kmax+\numax+1}$ with \eqref{eq:hankel_vec} \\
\For{$k=1,\ldots,\kmax+1$ }{
\nl Compute $r_k=aB_{\kmax+\numax}r_{k-1}$ where $B_{\kmax+\numax+1}$ is given by \eqref{eq:Br}\\
}
\nl Set $q=J^Te_1$ where $J=J_\mu\in\CC^{\kmax\times\kmax}$ given in \eqref{eq:jordan}\\
\For{$i=1,\ldots,p$}{
\nl Set $q=J^Tq-z_i q$
}
\For{$i=1,\ldots,\numax+1$}{
\nl Compute the tridiagonal Toeplitz matrix $H\in\CC^{\kmax\times \kmax}$ with \eqref{eq:f_jordan}
where we set $f^{(j)}(\mu):=(r_{j})_i$, $j=0,\ldots,\kmax$\\
\nl Compute the tridiagonal Toeplitz matrix $H'\in\CC^{\kmax\times \kmax}$  with \eqref{eq:f_jordan}
where we set $f^{(j)}(\mu):=(r_{j+1})_i$, $j=0,\ldots,\kmax$
\\
\nl Set $x_i=H'H^{-1}q$\\
}
\end{algorithm}

\begin{remark}{Efficiency improvements and memory requirements.}
Some further improvements in the computation of $x_0$ were
achieved by also using the localization of the vectors $q_i$. More precisely,
for $i>p+2$ the contribution of $A_1$ and $A_2$ do not need to be taken
into account.
Note that in order to carry out $m$ steps of TIAR for a NEP of size $N$, we mainly need to store a complex basis matrix of size $N\times m$. Hence, in our setting the memory required by TIAR is far less demanding than the storage required for the LU-factorization used for performing the initial step in \eqref{eq:x0_SMF}.
\end{remark}

\section{Numerical simulations} \label{sec:simulations}
In the numerical computations, we discretize \eqref{eq:eig_prob} with a finite element method and consider two geometries with smooth interfaces. Then  we approximate the eigenvalues of the matrix problem \eqref{eq:eig_DtN} with the new version of TIAR outlined in Section \ref{sec:TIAR}.
Note that in order to preserve the accuracy for basis functions of degree larger than one it is, for the considered geometries, mandatory to use curvilinear elements \cite{Ciarlet1972,brenner02}. When the eigenvalues are semi-simple, we expect optimal convergence with the $h$-version and with the $p$-version of the finite element method \cite{MR1115240}.  In $h$-FEM, we fix the polynomial order $p$ of the basis functions and decrease the maximum size of the cells $h$. Then, the optimal convergence is algebraic:
\begin{equation}
		\left|\frac{\lambda_j-\lambda_j^\fem}{\lambda_j}\right|\leq c N^{-p},\,\,c>0,
		\label{eq:hfem}
\end{equation}
where $N$ denotes the number of degrees of freedom. In $p$-FEM, we fix the mesh and increase $p$, which result in the exponential rate of convergence 
\begin{equation}
	\left|\frac{\lambda_j-\lambda_j^\fem}{\lambda_j}\right|\leq \alpha e^{ -\beta N^\frac{1}{2}},\,\,\alpha,\beta>0.
	\label{eq:pfem}
\end{equation}
It is well known that $hp$-FEM ($p$-FEM in the case of smooth interfaces) is superior to $h$-FEM in terms of accuracy vs number of degrees of freedom \cite{Schwab1998} but the use of higher order basis functions results in less sparse matrices.  Therefore, it is for a given matrix size more time consuming to solve a matrix eigenvalue problem generated by a high-order finite element method. 
The studied NEPs are solved with the specialization of the infinite Arnoldi method outlined in Section \ref{sec:TIAR}. We illustrate the convergence of the finite element method as well as the performance of the linear algebra solver. In particular, we show that a discretization with the $p$-version of the finite element method outlined in Section \ref{sec:Disc} together with the new version of the infinite Arnoldi method is an efficient tool for resonance calculations.
\begin{table}
\robustify\bfseries
\centering
\begin{tabular}{rrS[table-format=1.12, detect-weight]
                 S[table-format=1.12, detect-weight]
                 rrS[table-format=1.12, detect-weight]
                 S[table-format=1.12, detect-weight] }
\toprule
{$j$} & {$m$} & {$\Re{\lambda_j}$} & {$\Im{\lambda_j}$} & {$j$} & {$m$} & {$\Re{\lambda_j}$} & {$\Im{\lambda_j}$} \\
\midrule
1 &  1	 &  9.021766303207 & -0.273829280623 & 4 &  0	 & 19.243876046899 & -0.274713999601 \\
2 &  8	 &  8.936779164355 & -0.164935525246 & 5 & 19	 & 19.241527655113 & -0.104420737352 \\
3 & 14	 &  8.783835782061 & -0.000247588219 & 6 & 25	 & 19.156200970821 & -0.000653924318 \\
\bottomrule
\\
\end{tabular}

\caption{\emph{Selected reference eigenvalues for the problem \ref{sec:SD}, computed from \eqref{eq:reson} and ordered by $|\Im \lambda_j|$.}}
\label{tab:SD_reference}
\end{table}
\begin{figure}
	\centering
	\begin{tikzpicture}
		\draw( 0.00, 8.80) node {\includegraphics[scale=0.45]{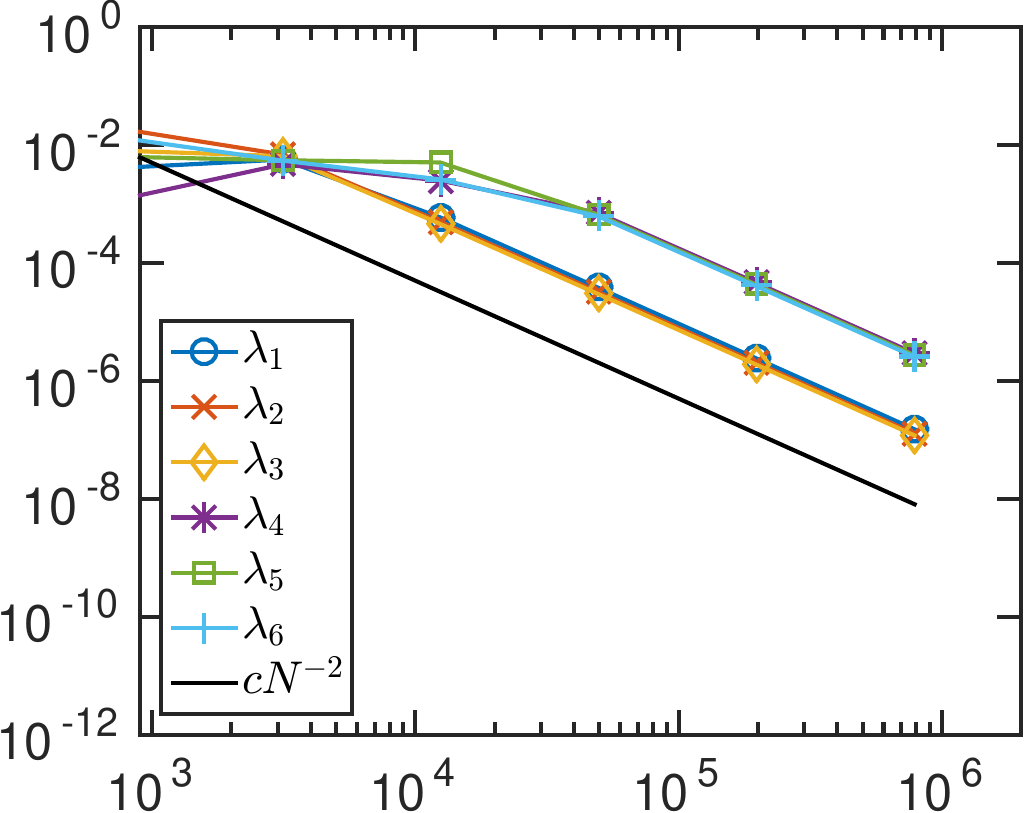}};
		\draw( 6.50, 8.71) node {\includegraphics[scale=0.45]{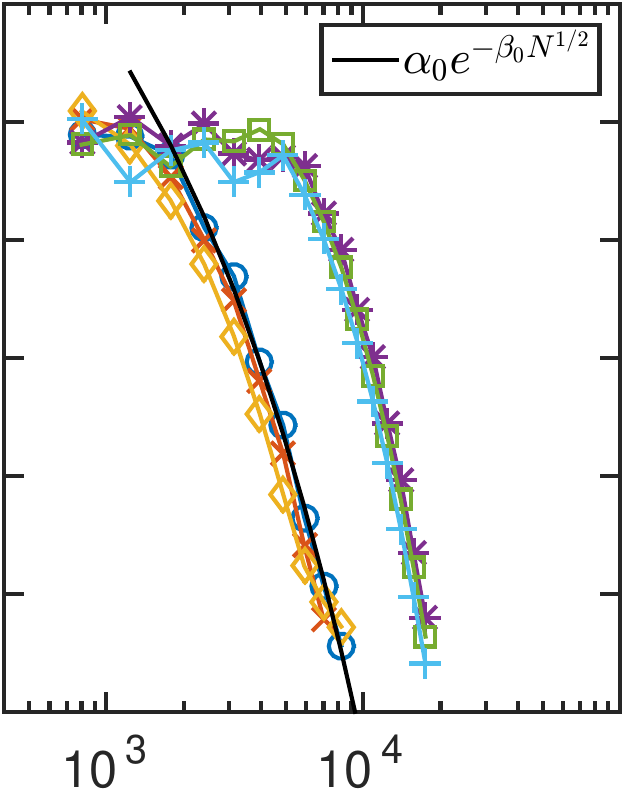} };
		\draw( 2.20, 1.90) node {\includegraphics[scale=0.45]{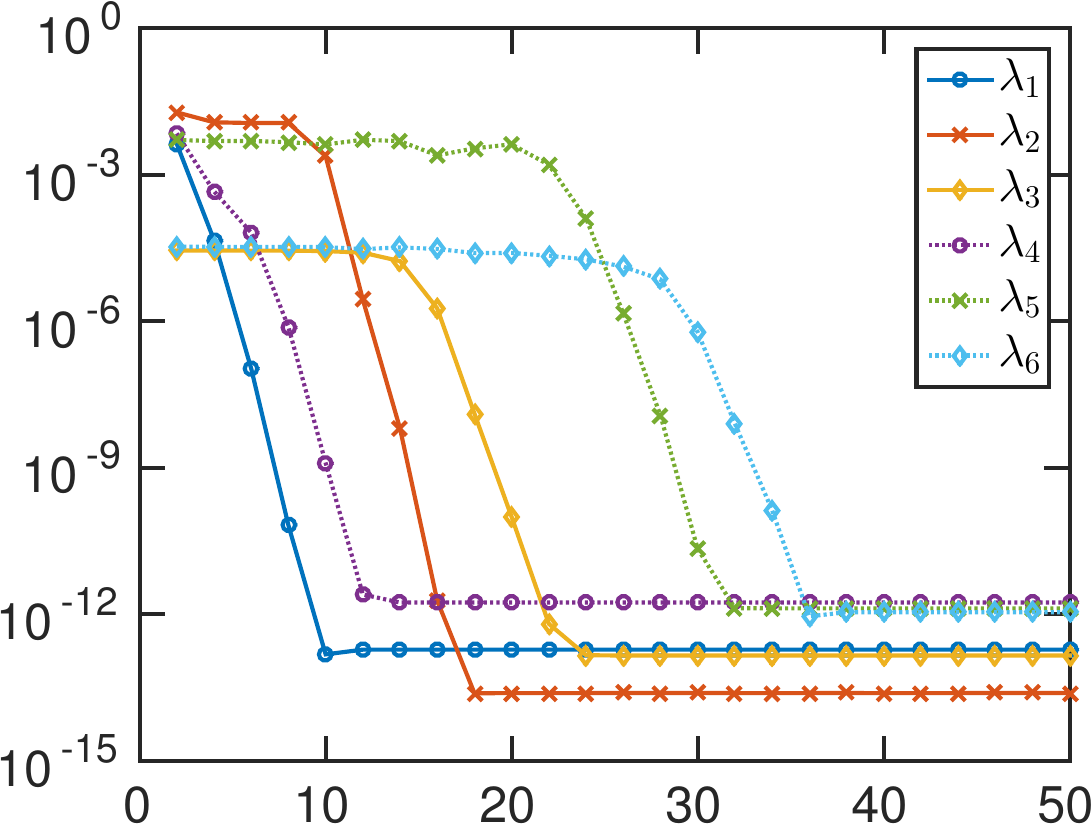}   };
		
		\node[text width=0.3cm] at ( -2.45, 11.20) {a)};
		\node[text width=0.3cm] at (  4.40, 11.20) {b)};
		\node[text width=0.3cm] at ( -0.60,  4.50) {c)};
		
		\node[text width=0.3cm] at ( 0.5,  5.5) {$N$};
		\node[text width=0.3cm] at ( 6.2,  5.5) {$N$};
		\node[text width=0.3cm] at ( 2.3, -1.5) {$\numax$};
		
		\node[text width=3.0cm, rotate=90] at ( -4.2,  9.4) {Relative error};
		\node[text width=3.0cm, rotate=90] at ( -2.3,  2.6) {Relative error};
	\end{tikzpicture}
	\caption{\emph{Convergence for problem \ref{sec:SD}: a) $h$-FEM with $p=2$ and b) $p$-FEM with fitting curve (black): $\alpha_0 = 1.22\times\,10^{5}, \beta_0 = 0.4081$. c) convergence with respect to $\numax$.}}\label{fig:SD_convergence}
\end{figure}
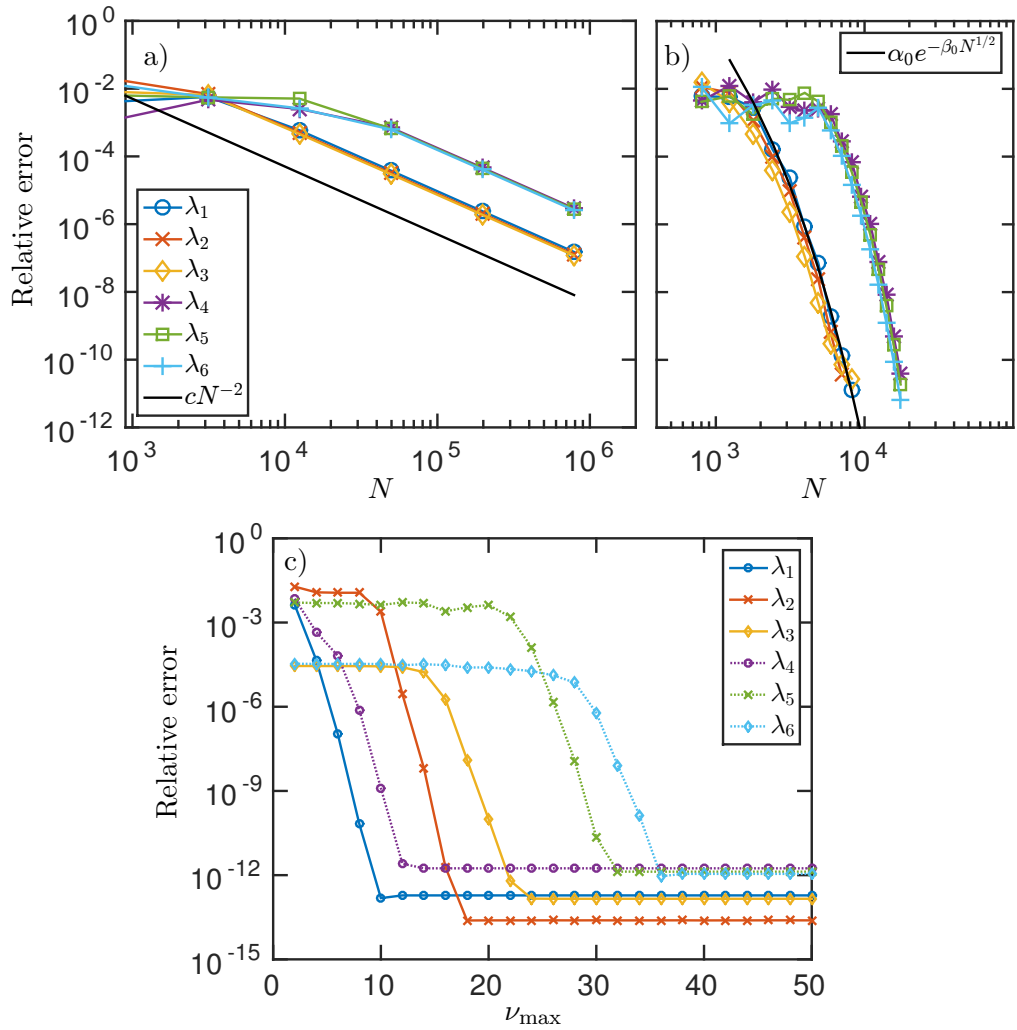
\subsection{Single disk problem}\label{sec:SD}
In this subsection, we consider the classical dielectric disk resonator \cite{dett09}. Using separation of variables, we obtain equation \eqref{eq:reson}. A complex Newton root finder is used to compute very accurate approximations of the resonances. These approximations are used as a benchmark for studying the convergence of the used finite element methods together with the new version of the infinite Arnoldi method.

Consider an open disk $\Om_R$ with radius $R$ and refractive index $\eta(r)=\eta_s$ in $\Om_R$ and $\eta(r)=1$ in $\Om_R^+:=\mx R^2\setminus\overline{\Om}_R$.
Exact solutions to this problem are given in \cite{dett09} and the resonance relationship reads 
\begin{equation}
J_m (\lambda \eta_s R) (H_m^{(1)} (\lambda R))'-
\eta_s J'_m(\lambda \eta_s R) H_m^{(1)} (\lambda R)=0.
\label{eq:reson}
\end{equation}
For each $m$ in equation \eqref{eq:reson}, we search numerically $n$ resonances $\lambda_{m 1},\ldots,\lambda_{m n}$. The Newton root finder presented in \cite{Yau98} is used with machine precision stopping criterion.

In Table \ref{tab:SD_reference}, we list a selection of benchmark values $\lambda_{m n}$ computed from \eqref{eq:reson}, which are used to evaluate accuracy and convergence of the numerical solutions. 

The solutions $\lambda_{m n}$ are classified as \emph{external resonances} or \emph{internal resonances} \cite{dett09}. The internal resonances, also called \emph{Whispering Gallery Modes} (WGM),
have broad applications in optics, photonics, communications, and engineering \cite{Ilch06}. These resonances feature negative imaginary parts that are close to the real axis. Approximation of the internal modes with FEM require that the oscillatory behavior inside the resonator is resolved, but the modes outside the resonator are almost constant and therefore less demanding to approximate. 

Exterior resonances are characterized by having large negative imaginary parts and the corresponding modes grow very quickly outside the resonator. 
Hence, FEM approximation of these modes is demanding.
\begin{table}
\robustify\bfseries
\centering
\begin{tabular}{rS[table-format=3.10, detect-weight]
                 S[table-format=1.13, detect-weight]
                rS[table-format=3.10, detect-weight]
                 S[table-format=1.14, detect-weight] }
\toprule
{$j$} & {$\re \lambda_j$} & {$\im \lambda_j$} & {$j$} & {$\re \lambda_j$} & {$\im \lambda_j$} \\
\midrule
1  &  3.499842     & - 8.4003189     & 15 & 19.2563064489 & - 0.06652895684  \\
2  &  3.082426     & - 8.1795102     & \bftab 16 &  99.0706091289 &  -0.53389807975   \\
3  &  3.662856165  & - 4.980016551   & \bftab 17 &  98.6967997099 &  -0.386922639003  \\
4  &  3.035882038  & - 4.910209072   & 18 & 98.8356759636 & - 0.069143373791  \\
5	 &	1.0986166110 & - 1.00574509569 & 19 & 98.8254516105 & - 0.059689390790  \\
6	 &  2.1655203793 & - 0.53731229013 & 20 & 99.4061929966 & - 0.000006880063  \\
7  &  4.3700360826 & - 1.52656168626 & 21 & 99.4061923466 & - 0.000006849959  \\
8  &  3.9580686857 & - 0.52895955645 & 22 & 99.4061921097 & - 0.000009251593  \\
9  &  4.8949905287 & - 0.40281083717 & \bftab 23 &  99.4061941089 &  -0.000009059285  \\
10 & 20.0018652230 & - 1.19979332765 & 24 & 99.2530774600 & -0.000000004940  \\
\bftab 11 &  19.1590173402 &  -0.63161790252  & 25 & 99.2530774626 & - 0.000000004465  \\
12 & 20.3296169084 & - 0.53046501438  & 26 & 99.2530774593 & - 0.000000004441  \\
13 & 21.0596179198 & - 0.41347266647  & 27 & 99.2530774637 & -0.000000004263  \\
\bftab 14 &  19.1765650857 &  -0.08415304732  & \bftab 28 &  99.2229411961 &  -0.0000000000001 \\
\bottomrule
\end{tabular}

	




\caption{\emph{Reference eigenvalues for problem \ref{sec:TD}, computed with with $p=30$, $N$=$317\,281$.}}
\label{tab:TD_reference}
\end{table}
\subsubsection{Approximation with FEM and TIAR}
If the DtN map \eqref{eq:DtN} is placed at $a=R$, then only one term is non-zero because the DtN map coincides with the compatibility condition for derivatives. However, by shifting the disk $\Omega_R$ a distance $d$ away from the origin and taking $a>R+d$, the radial symmetry is lost and more terms are necessary in the DtN map \eqref{eq:DtN}. This more demanding problem formulation is used to test our solution scheme in terms of accuracy as well as convergence with respect to the number of DtN terms. Let $\Gamma_a:=\partial\Omega_a$ denote a circle of radius $a$ centred in the origin and let $\Omega^d_R\subset\Omega_a$ denote the shifted disk. We choose $R=1$ with the shifting distance $d=R/2$ along the $x$-axis, $\eta_s=2$, and $a=2$.

Once the geometry is set, we discretize \eqref{eq:eig_prob} with the finite element method outlined in Section \ref{sec:DtN_2d} and apply the new specialization of TIAR in Section \ref{sec:TIAR} to a FE discretization with $p=20$, $a=2$ and $N=19\,361$. 
In the TIAR scheme we use the shifts $\mu=9-0.1i$ to approximate $\lambda_{1}, \lambda_{2}, \lambda_{3}$ and $\mu=19-0.1i$ for $\lambda_{4}, \lambda_{5}, \lambda_{6}$, where $\lambda_1,\ldots,\lambda_6$ are the reference values in Table \ref{tab:SD_reference}. The resonances $\lambda_{3}$ and $\lambda_{6}$ are classified as interior resonances while the remaining values are classified as exterior resonances \cite{dett09}. For the given shifts, we compute simultaneously approximations to both types of resonances. Hence, the used finite element space must be able to capture oscillations in the interior of the resonator and rapid growth in the exterior.

Figure \ref{fig:SD_convergence} illustrates the convergence of a sequence of eigenvalues $\lambda_j^\fem$ approaching the reference values $\lambda_j$ in Table \ref{tab:SD_reference}. Optimal convergence rates are reached: \eqref{eq:hfem} in a) and \eqref{eq:pfem} in b).
For the analysis of the $\numax$-dependence we work with a fixed discretization and plot in Figure \ref{fig:SD_convergence} (c) relative errors vs $\numax$. We observe a preasymptotic phase until a critical $\numax=\tilde \nu$ is reached and for $\numax > \tilde \nu$ the convergence is very rapid. Figure \ref{fig:SD_convergence} (c) also illustrates stability of TIAR in the sense that having more $\nu$ terms than $\tilde \nu$ in the numerical experiments did not result in larger numerical errors.

As expected, errors in the approximations $\lambda_j^\fem$ drop faster when $\lambda_j$ in Table \ref {tab:SD_reference} correspond to an eigenvalue with small $m$. The empirical rule $\numax>a\lambda$ was used in \cite{Harari92} as an estimate of the necessary number of truncation terms $\numax$ in the DtN method applied for source problems ($\lambda\in\mx R$). Our computations indicate that $\numax>a\re\lambda$ can be used for the current eigenvalue problem. For example, from Figure \ref{fig:SD_convergence} (c), it is clear that the relative errors of the numerical $\lambda_j^\fem$ decrease below $10^{-9}$ for $\numax > a\re \lambda_j^\fem$.
\begin{figure}
	\centering
	\begin{tikzpicture}
		\draw( 0.00, 4.5) node { \includegraphics[scale=0.108]{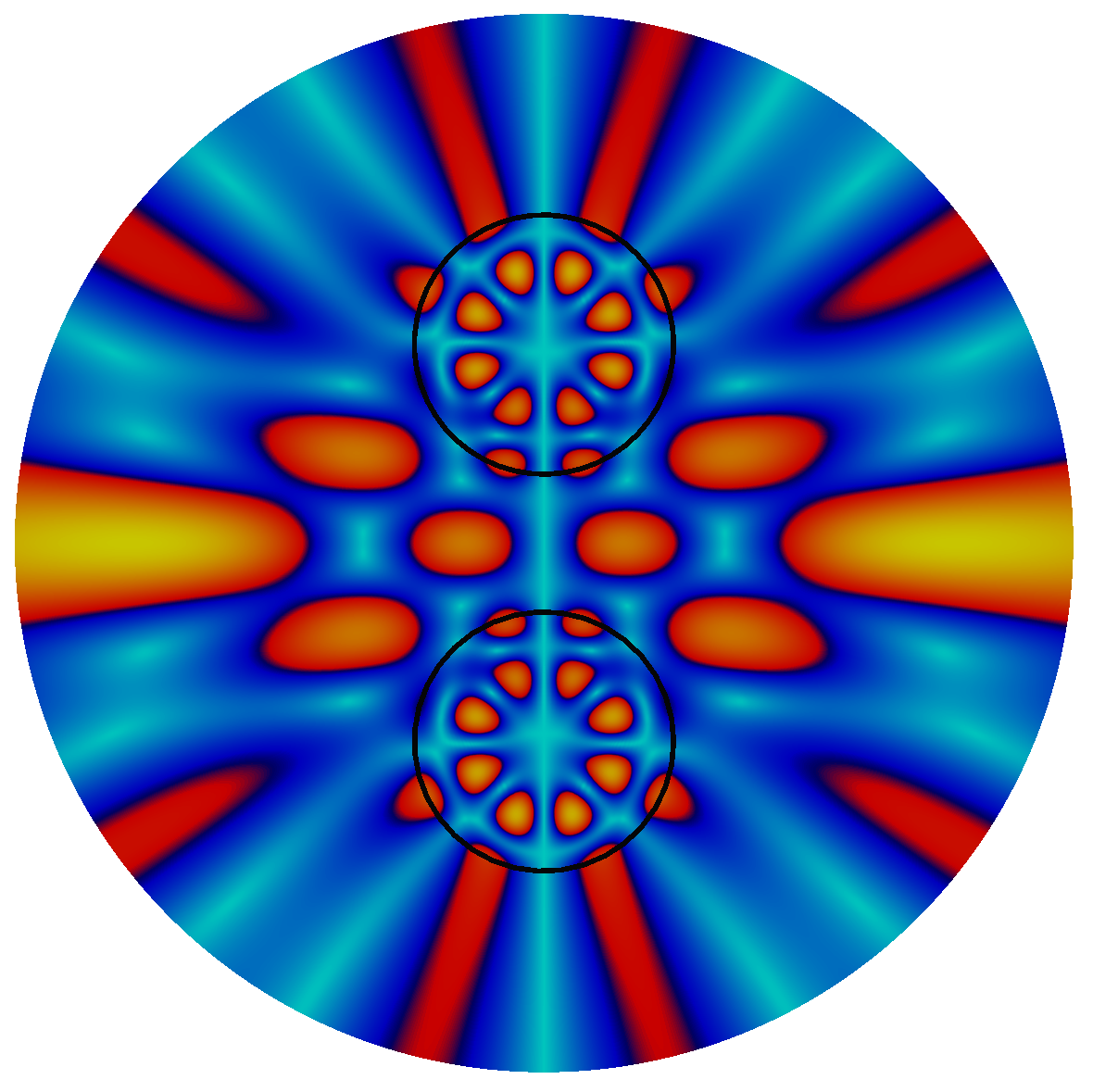} }; 
		\draw( 4.50, 4.5) node { \includegraphics[scale=0.108]{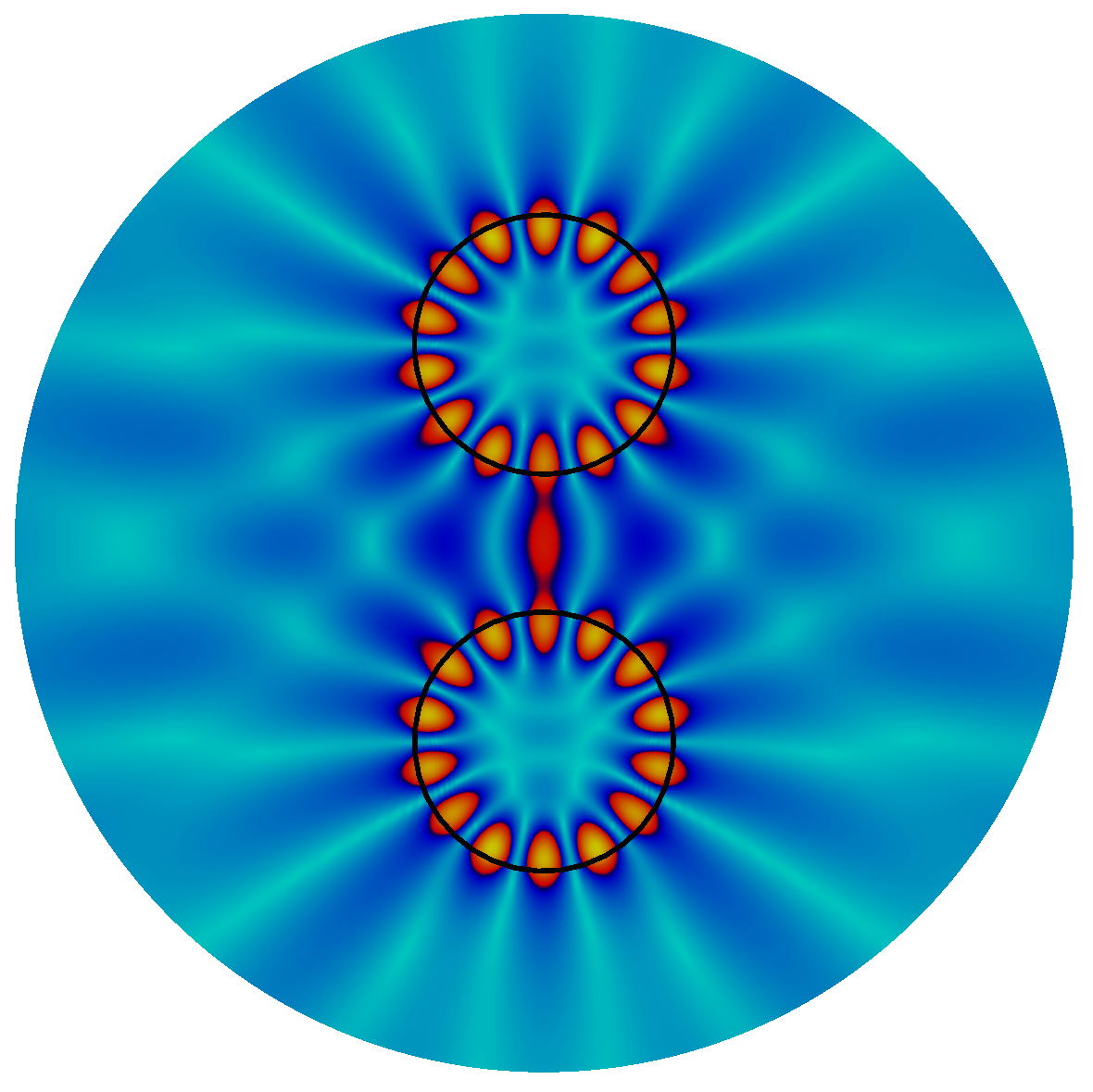} }; 
		\draw( 9.00, 4.5) node { \includegraphics[scale=0.108]{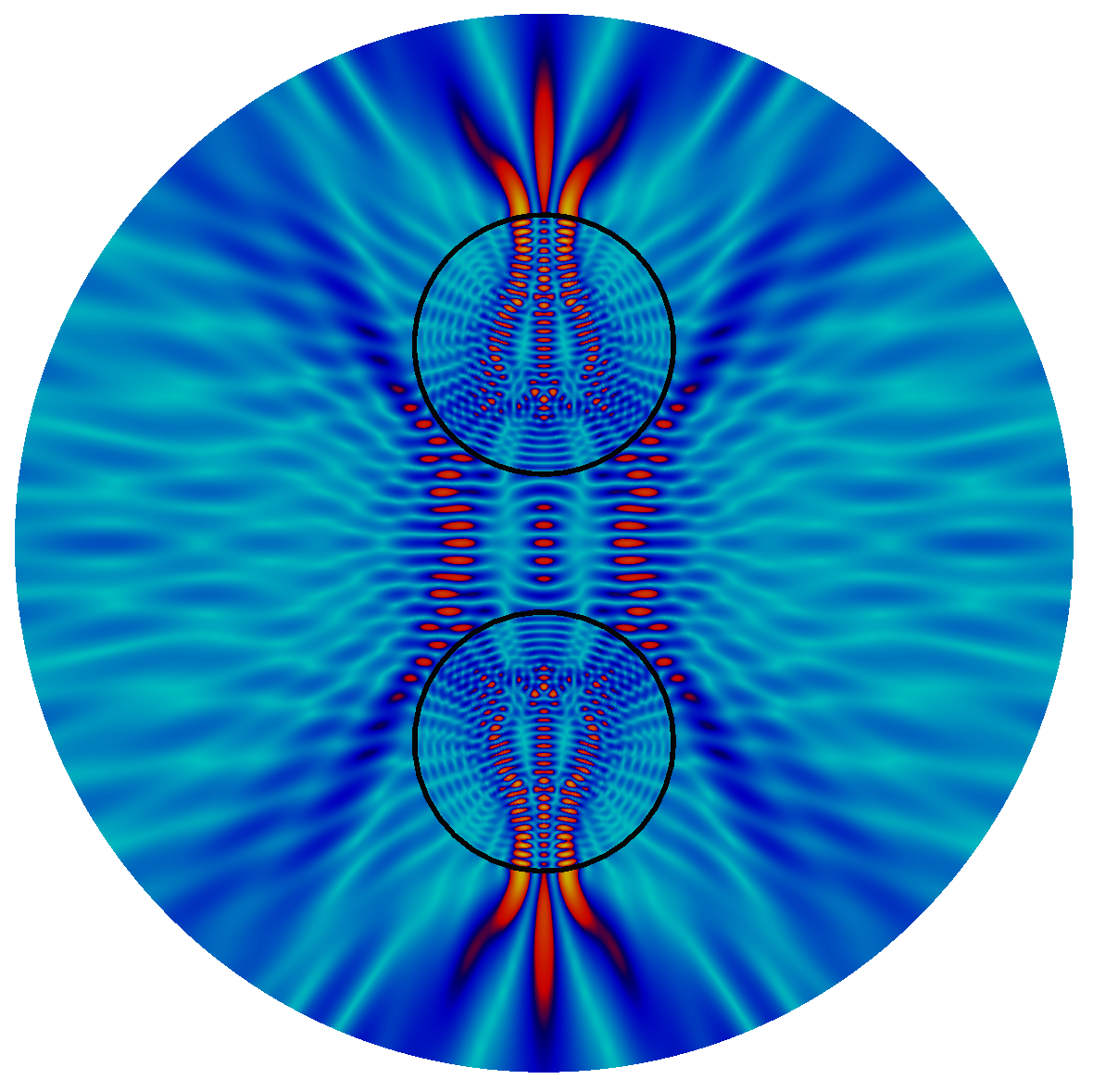} }; 
		
		\draw( 0.00, 0.0) node { \includegraphics[scale=0.108]{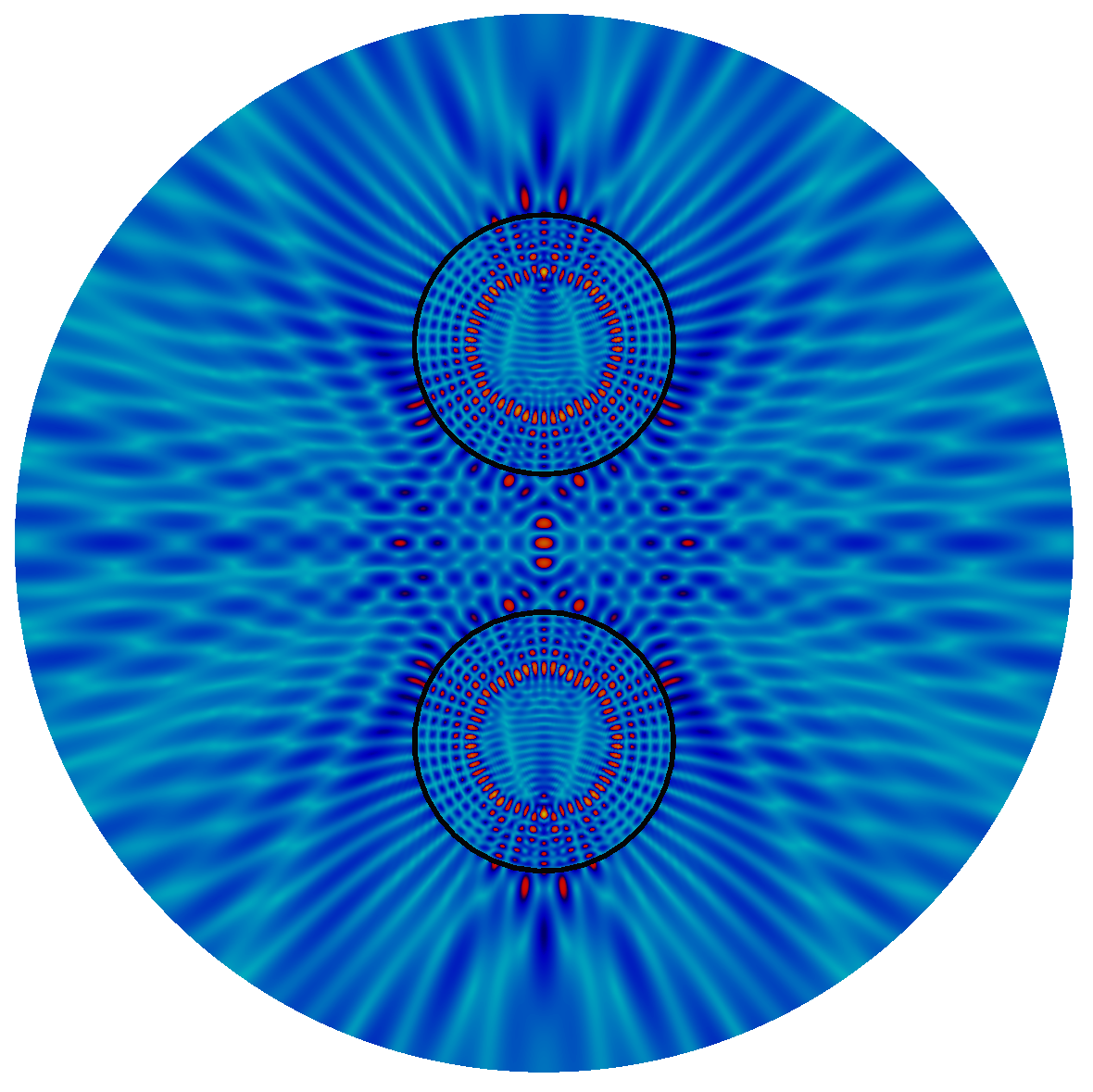} }; 
		\draw( 4.50, 0.0) node { \includegraphics[scale=0.108]{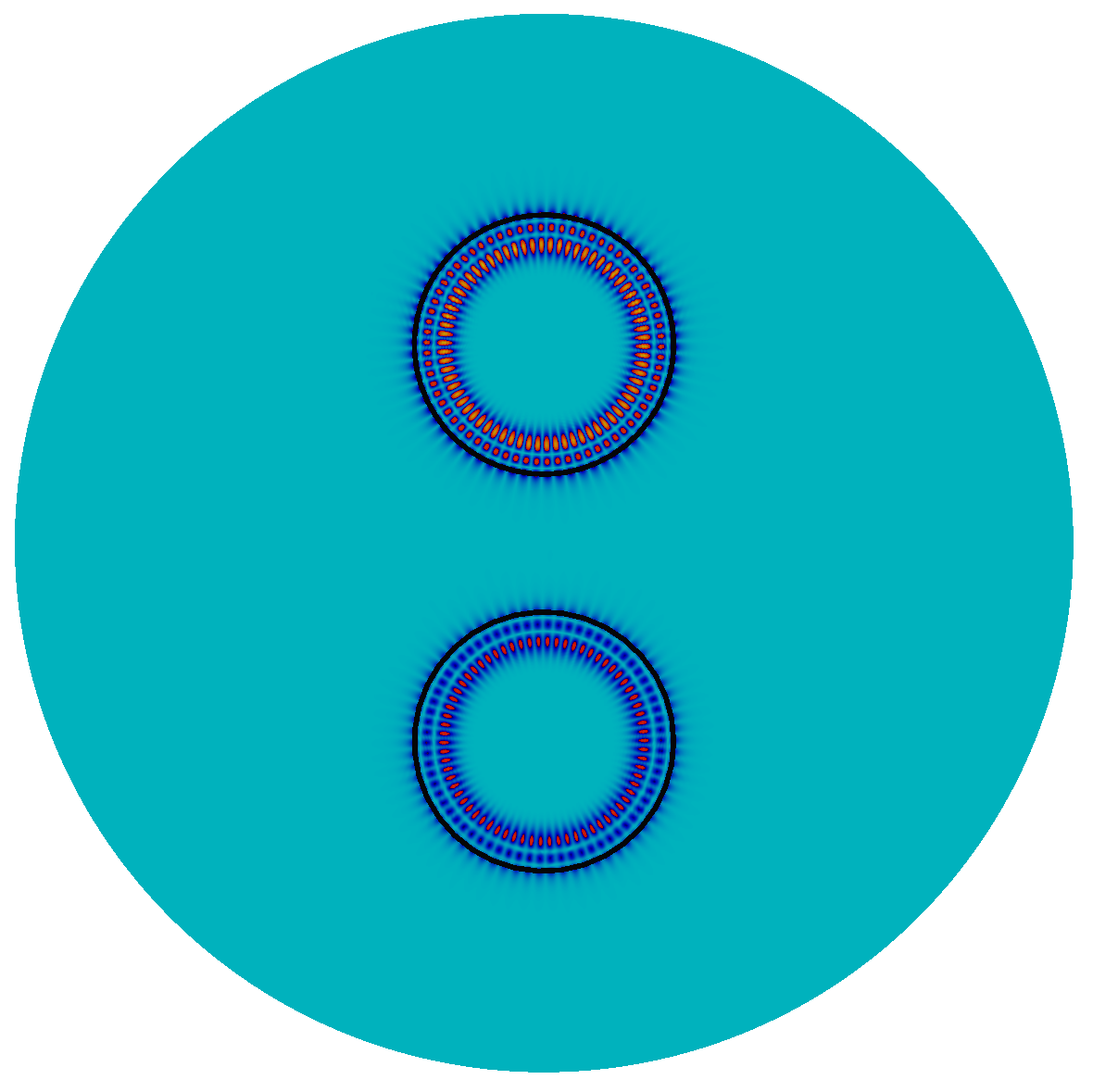} }; 
		\draw( 9.00, 0.0) node { \includegraphics[scale=0.108]{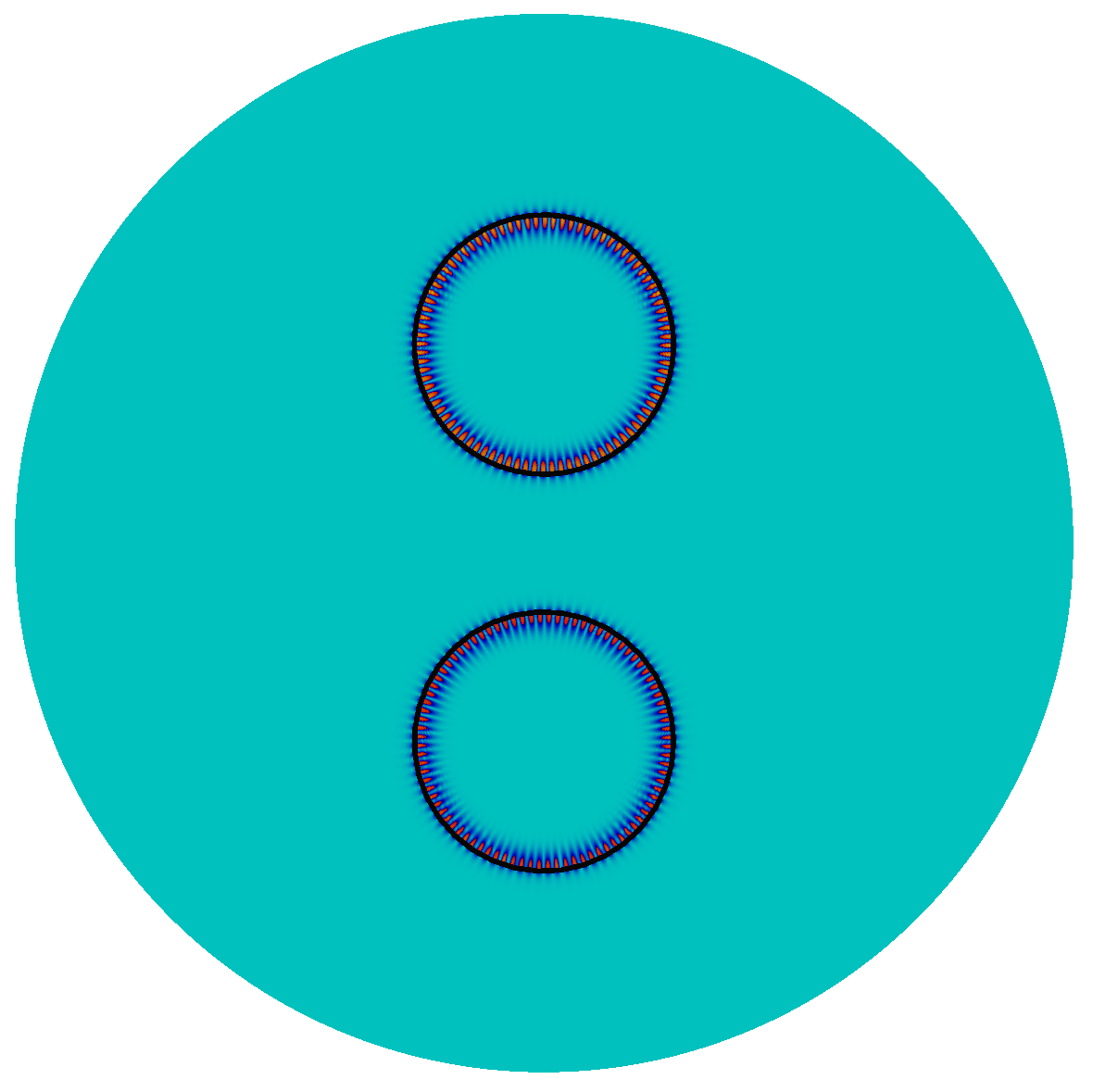} }; 
		
		\node[text width=0.3cm] at ( 0.0-2.0, 4.5+2.0) {{\small $\lambda_{11}$}};
		\node[text width=0.3cm] at ( 4.5-2.0, 4.5+2.0) {{\small $\lambda_{14}$}};
		\node[text width=0.3cm] at ( 9.0-2.0, 4.5+2.0) {{\small $\lambda_{16}$}};
		\node[text width=0.3cm] at ( 0.0-2.0, 0.0+2.0) {{\small $\lambda_{17}$}};
		\node[text width=0.3cm] at ( 4.5-2.0, 0.0+2.0) {{\small $\lambda_{23}$}};
		\node[text width=0.3cm] at ( 9.0-2.0, 0.0+2.0) {{\small $\lambda_{28}$}};
		
	\end{tikzpicture}
	\caption{\emph{Fields $|u_j(x_1,x_2)|$ corresponding to eigenvalues $\lambda_j$ given with bold font in Table \ref{tab:TD_reference}.	}}
	\label{fig:TDP_functions}
\end{figure}
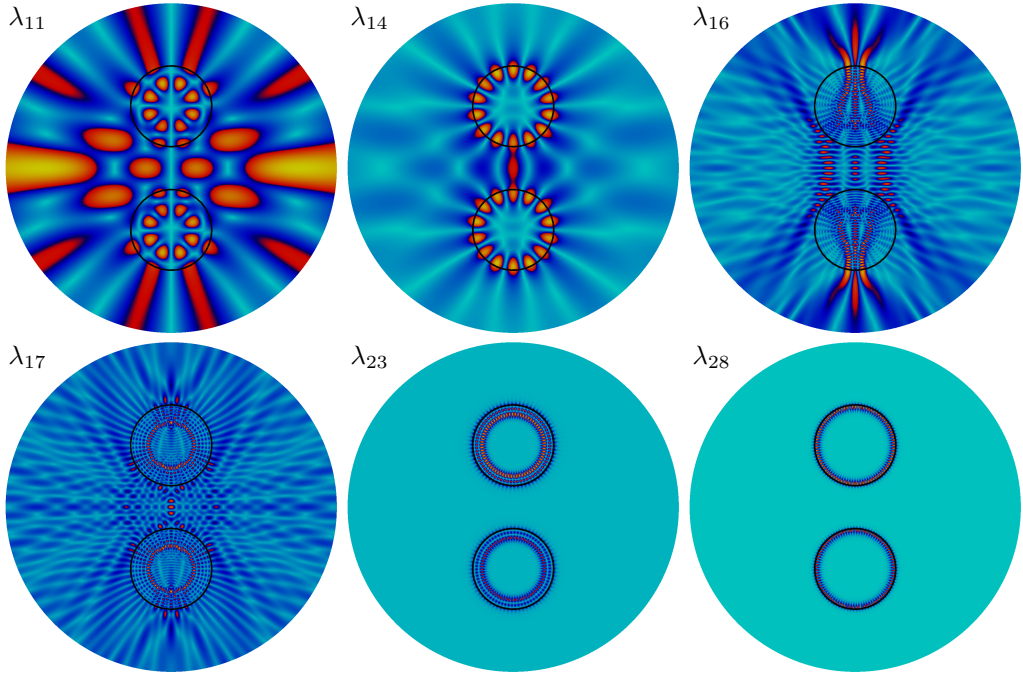
\subsection{Disk dimer problem}\label{sec:TD}
Adjacent resonators are of particular interest as they exhibit oscillatory modes that cannot be excited by single resonators and have interesting physical properties such as special mode symmetries (fanoresonances) that may exhibit strong coupling, and higher $Q$-factors \cite{Hump16,Ros15,Mort14}. 

The studied geometry consists of two disks of radius $R=1/4$ separated vertically by a distance $s=R$. Each disk has constant refractive index $\eta=2$, and are surrounded by vacuum $\eta=1$. The setting is such that $\hbox{supp}(\eta^2-1)\subset\Omega_a$ with $a=1$. We study convergence by computing reference eigenvalues $\lambda_j$ with the new specialization of TIAR, outlined in section \ref{sec:TIAR}, applied to a discretization with $p=30$ and list them in Table \ref{tab:TD_reference}.
\\
Since no exact solution is available, we study convergence with respect to three different parameters: degrees of freedom $N$, $\nu$ and TIAR iterations $k$.

\subsection{Approximation with FEM and TIAR}\label{sec:app_TD}
In Figure \ref{fig:TD_convergence} we show convergence for the disk dimer problem. In (a) and (b) convergence with respect to $N$ and in c) convergence with respect to $\nu$. The computed relative errors of the eigenvalues satisfy the estimates \eqref{eq:hfem} and \eqref{eq:pfem}. Hence, the numerical computation indicate that the asymptotic convergence rates are optimal. As expected, the preasymptotic phase is longer for large $\re \lambda_j$ \cite{Sauter10} and eigenvalues with smaller $\re\lambda_j$ converge faster. This can be seen by comparing the model \eqref{eq:pfem} for different $\lambda^\fem_j$. The fitted curves following $\lambda_{11}^\fem$ and $\lambda_{17}^\fem$ are plotted in dashed-black and solid-black lines in Figure \ref{fig:TD_convergence} (b). The exponential rate for $\lambda_{11}$ is $\beta_1=0.1799$ and for $\lambda_{17}$ is $\beta_2=0.1203$, then $\beta_1>\beta_2$ is in agreement with $\re \lambda_{11}=19.159<\re \lambda_{17}=98.697$.

The convergence of relative error with respect to $\numax$ behaves similarly as described in Section \ref{sec:SD}. 
In Figure \ref{fig:TD_convergence} (c), the real part of $\lambda_{11}$ is $19.159$ and the relative error drops below $10^{-8}$ for $\numax>20$. Similarly for $\re\lambda_{16}=99.0706$, the relative error is below $10^{-10}$ for all $\numax>80$. 

Figure \ref{fig:TDP_functions} shows that the function $u_{28}$ corresponding to $\lambda_{28}$ oscillates in a confined region around the resonator resembling WGMs \cite{Ilch06}. Furthermore, the figure shows that $|u_{28}(x)|\approx\,\,$constant for $x\in\Gamma_a$. Hence, in the Fourier series \eqref{eq:dtn_1} only the $\nu=0$ term is necessary to accurately approximate $u_{28}$. In agreement, Figure \ref{fig:TD_convergence} (c) shows that the relative error in $\lambda_{28}$ is approximately $10^{-9}$ for $\numax\geq 0$.
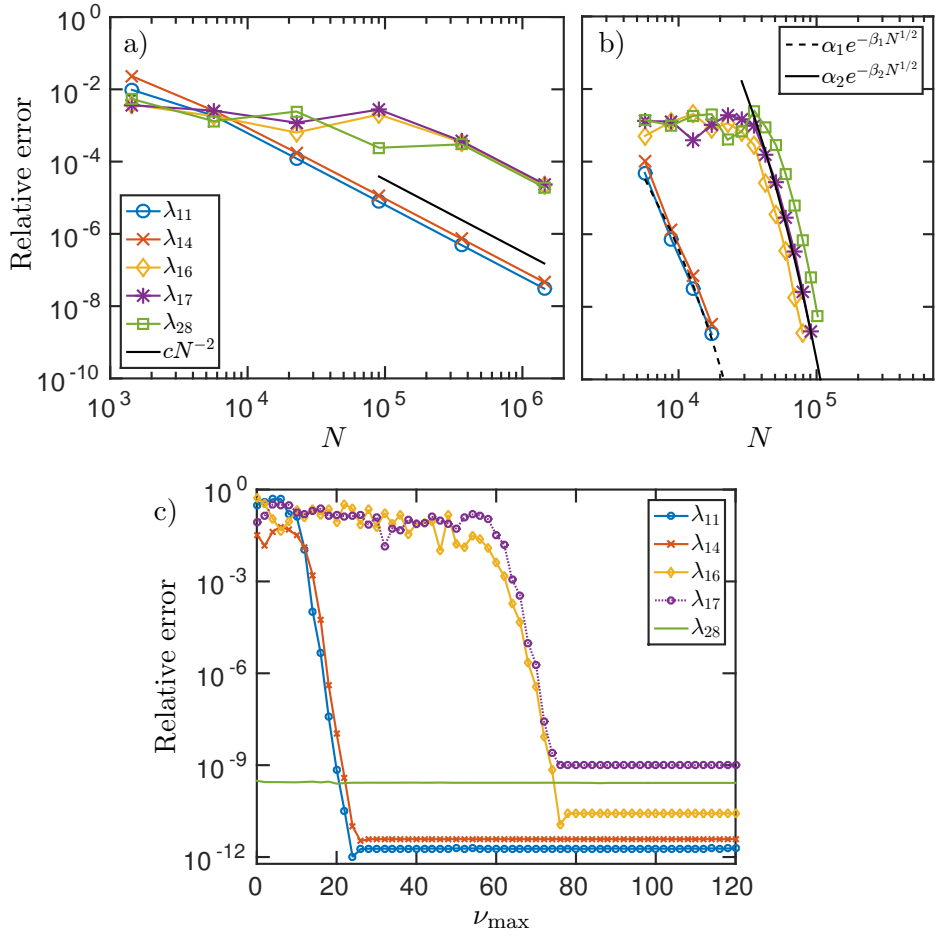
\begin{figure}
	\centering
	\begin{tikzpicture}
	
		\draw( 0.00, 7.50) node { \includegraphics[scale=0.40]{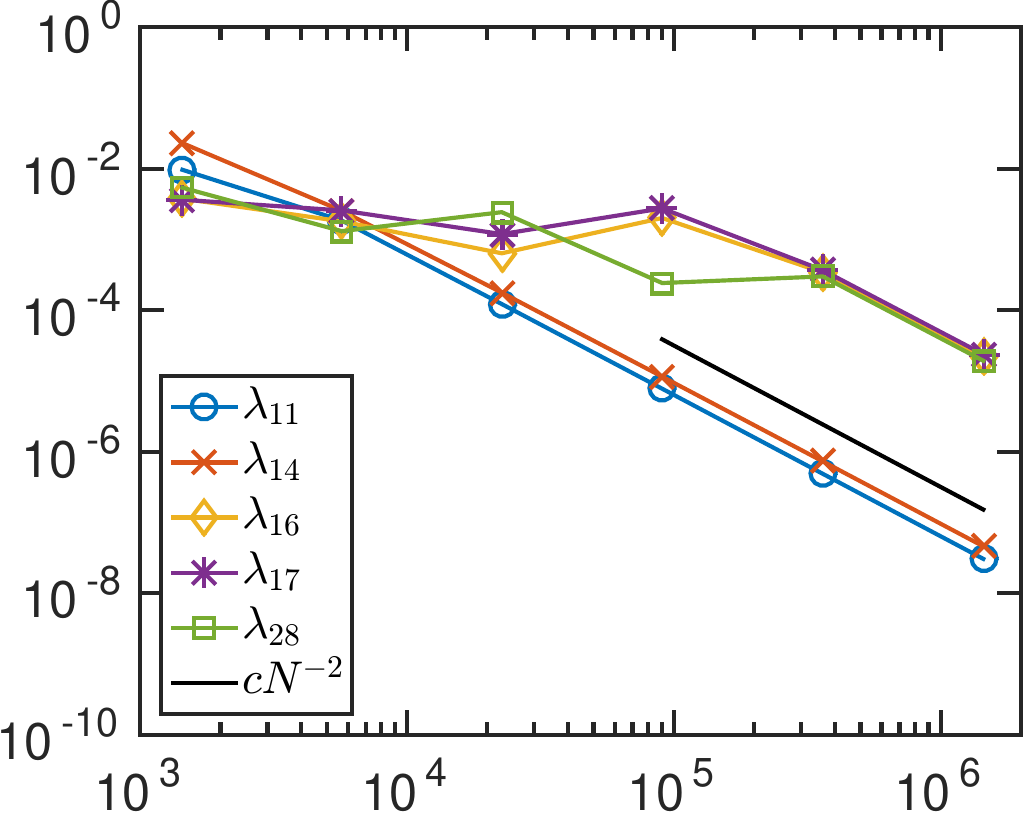}};
		\draw( 6.00, 7.42) node { \includegraphics[scale=0.40]{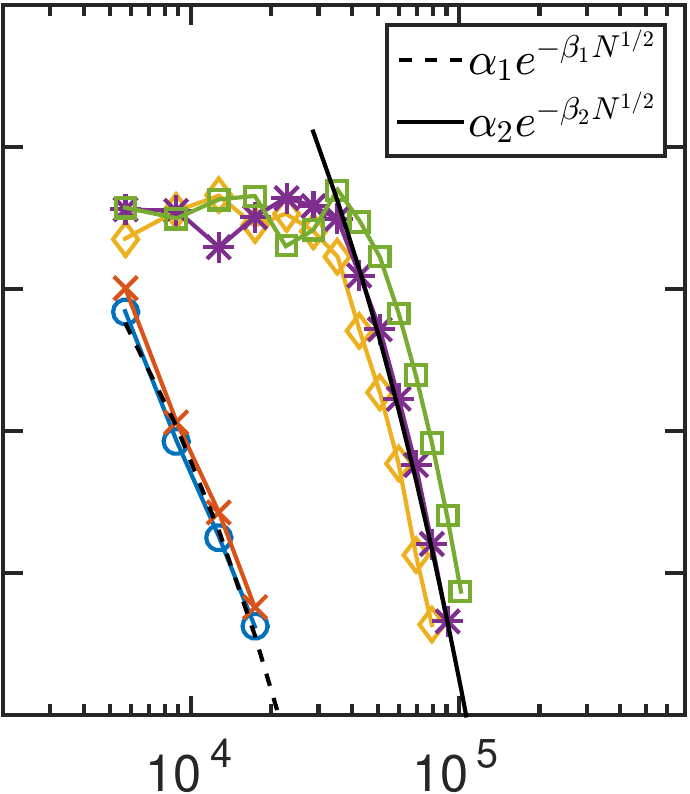} };
		\draw( 2.20, 1.20) node { \includegraphics[scale=0.40]{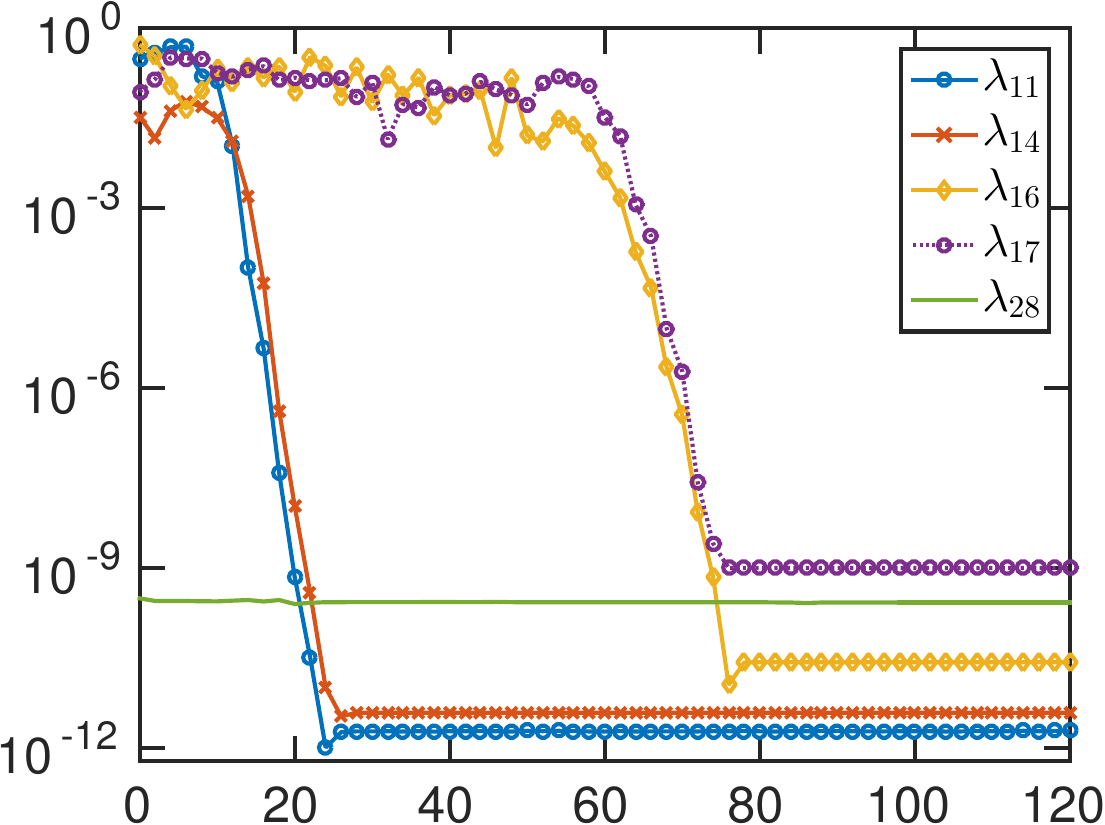}   };
		
		\node[text width=0.3cm] at ( -2.2, 9.7) {a)};
		\node[text width=0.3cm] at (  4.05, 9.7) {b)};
		\node[text width=0.3cm] at ( -1.8, 3.5) {c)};
		
		\node[text width=0.3cm] at ( 0.4,  4.5) {$N$};
		\node[text width=0.3cm] at ( 6.0,  4.5) {$N$};
		\node[text width=0.3cm] at ( 2.4, -1.9) {$\numax$};
		
		\node[text width=3.0cm, rotate=90] at ( -3.7,  8.2) {Relative error};
		\node[text width=3.0cm, rotate=90] at ( -1.8,  1.8) {Relative error};
		
	\end{tikzpicture}
	\caption{\emph{Convergence for problem \ref{sec:TD}: a) $h$-FEM with $p=2$ and b) $p$-FEM. The fitting curves in black use: 
$\alpha_1 = 26.46,    \beta_1 = 0.1799,
\alpha_2 = 1.14\times\,10^{7}, \beta_2 = 0.1203$. c) convergence with respect to $\numax$.
	}}
	\label{fig:TD_convergence}
\end{figure}

In Figure \ref{fig:cancel} (a) we present a selection of computed eigenvalues of equation \eqref{eq:eig_DtN} corresponding to the disk dimer problem. The eigenvalues marked with red stars are listed in Table \ref{tab:TD_reference}. The blue bullets correspond to the poles $z_j$ of $G(\lambda)$ defined in \eqref{eq:def_fem} with $a=1$. In the Figure \ref{fig:cancel} (b) we illustrate a situation with $a=2$, where there are poles in between the shift $\mu$ and the closest $\lambda^\fem_j$. In this case a pole gets very close to $\lambda_{5}$ as illustrated in the Figure \ref{fig:cancel} (b). We evaluate the effectiveness of the pole cancellation technique by choosing $\mu=1.15-0.8i$ and running TIAR with and without pole cancellation. In Figure \ref{fig:cancel_ite} (b) we show the convergence vs. iterations of the experiment described above. Without pole cancellation we only get convergence for $\lambda_5$ (red dashed line), until stagnation at around $10^{-5}$. When using pole cancellation, $\tilde\lambda_5$ converges until machine precision (solid line) and $\tilde\lambda_6$ also converges (dotted line), which suggests that the radius of convergence is greater when pole cancellation is used. \\

\begin{figure}
	\centering
	\begin{tikzpicture}
		\draw( 0.00, 2.7) node { \includegraphics[scale=0.35]{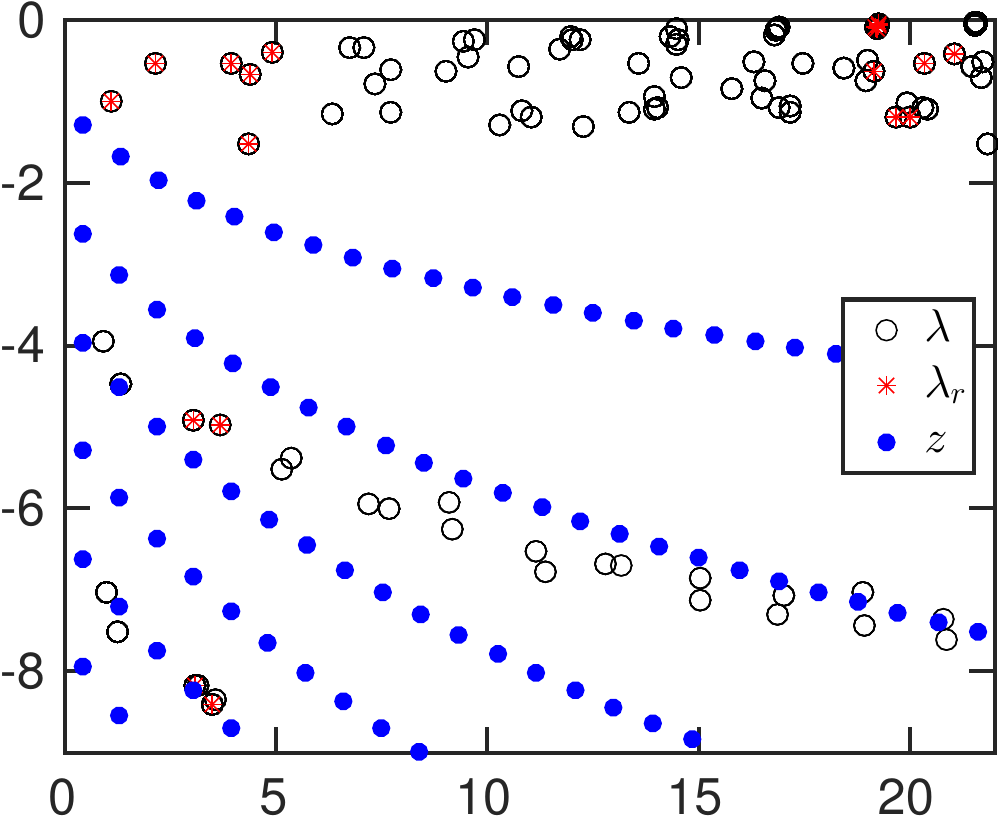} };
		\draw( 6.60, 2.7) node { \includegraphics[scale=0.35]{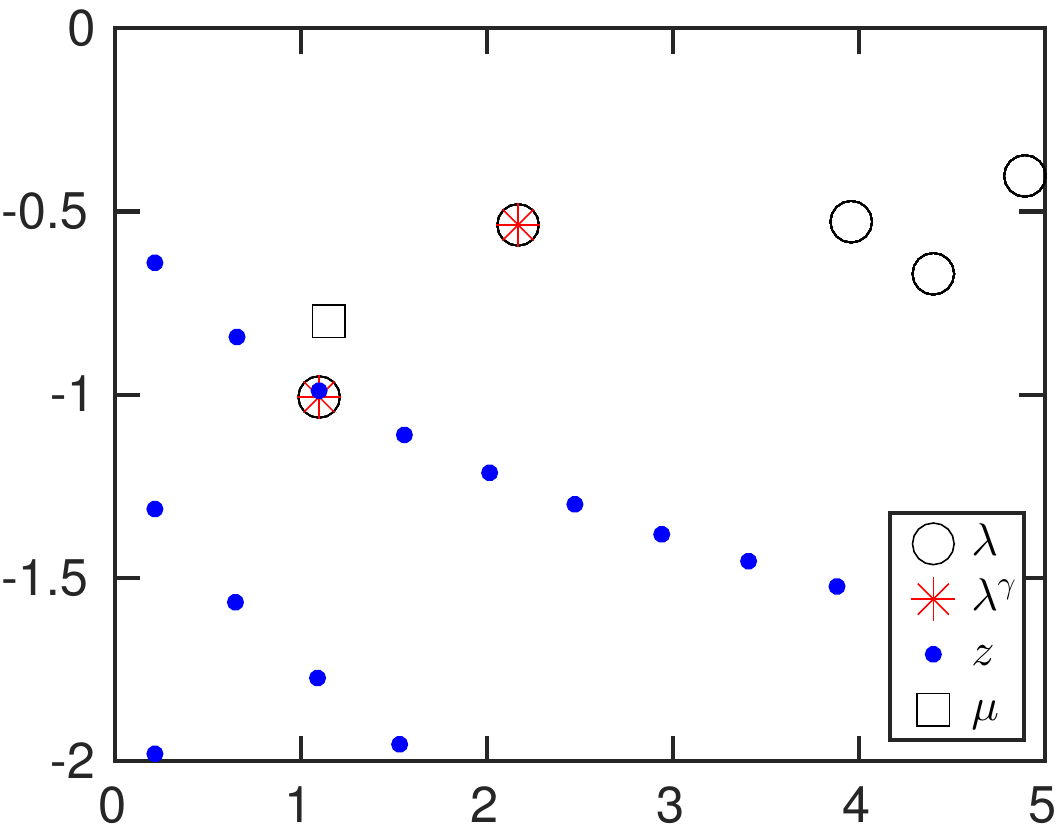}   };
		\node[text width=5.3cm] at ( 0.20, 5.30) {a)};
		\node[text width=5.3cm] at ( 6.90, 5.30) {b)};
		\node[text width=0.5cm] at ( 5.20, 2.58) {{$\lambda_{5}$}};
		\node[text width=0.5cm] at ( 6.95, 3.82) {{$\lambda_{6}$}};
		\node[text width=1.3cm] at ( 0.35, 0.0) {{ $\re \lambda$}};
		\node[text width=1.3cm] at ( 7.15, 0.0) {{ $\re \lambda$}};
		\node[text width=5.3cm] at ( -1.0, 2.8) {{$\im \lambda$}};
	
	\end{tikzpicture}
	\caption{\emph{a) Eigenvalues $\lambda$ computed with several shifts $\mu$ and $a=1$. 
	Selected reference eigenvalues $\lambda_r$, included in Table \ref{tab:TD_reference}, are marked with red stars and poles $z$ in blue bullets. b) Situation illustrating pole cancellation with $a=2$.}}
	\label{fig:cancel}
\end{figure}

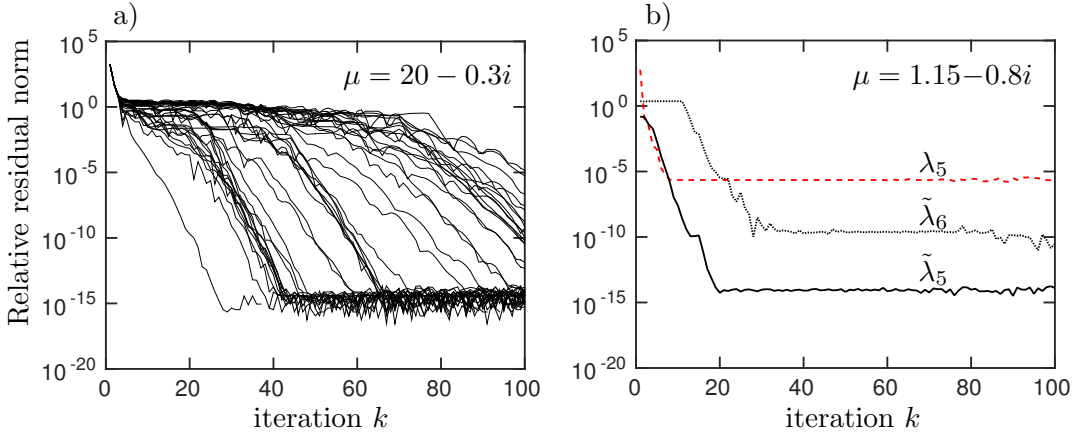
\begin{figure}
	\centering
	\begin{tikzpicture}
		\draw( 0.00, 1.0) node { \includegraphics[scale=0.35]{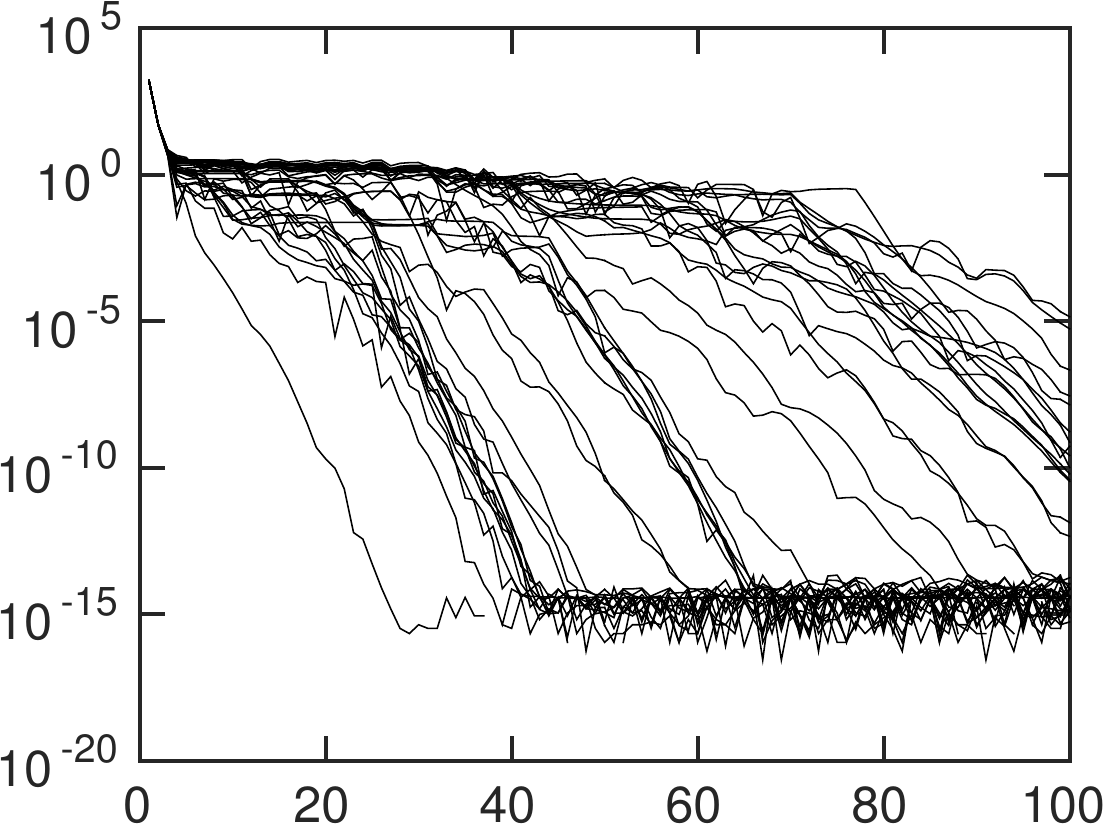} };
		\draw( 7.00, 1.0) node { \includegraphics[scale=0.35]{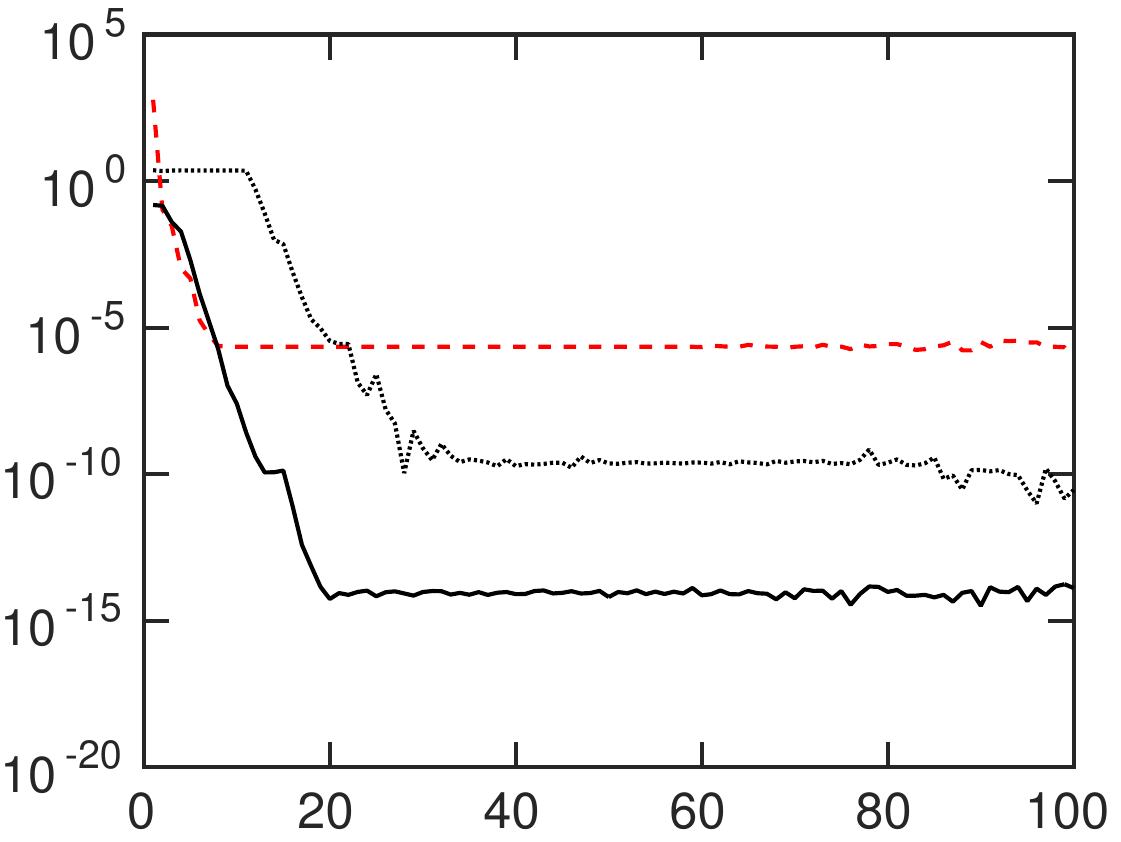} };
		\node[text width=5.3cm] at ( 0.30, 3.60) {a)};
		\node[text width=5.3cm] at ( 7.30, 3.60) {b)};
		\node[text width=0.5cm] at ( 8.50, 0.25) {{$\tilde\lambda_{5}$}};
		\node[text width=0.5cm] at ( 8.50, 1.0) {{$\tilde\lambda_{6}$}};
		\node[text width=0.5cm] at ( 8.50, 1.65) {{$\lambda_{5}$}};
		\node[text width=2.5cm] at ( 1.80, 2.80) {{ $\mu=20-0.3i$ }};
		\node[text width=2.5cm] at ( 8.50, 2.80) {{ $\mu=1.15-0.8i$}};
		
		\node[text width=4.0cm, rotate=90] at ( -3.6, 1.5) {Relative residual norm};
		
		\node[text width=3.0cm] at ( 1.00, -1.7) {iteration $k$};
		\node[text width=3.0cm] at ( 8.00, -1.7) {iteration $k$};
		
	\end{tikzpicture}
	\caption{\emph{a) TIAR eigenvalue convergence vs iterations. b) Case illustrating Pole cancellation. We show convergence without pole canceling in red dashed line and convergence using pole canceling in solid line.}}
	\label{fig:cancel_ite}
\end{figure}

\noindent\emph{Performance comparison:} Below we briefly discuss the performance of the proposed NEP solver. In the startup phase, we compute for given shift $\mu$ the LU factorization of $T(\mu)$ and refer to the time spent as LU time. The LU factorization is only computed once and it is used in the inverse operation \eqref{eq:x0_SMF}. 

In Figure \ref{fig:TD_performance} we show performance plots for these routines and give in colors the relative error of $\lambda_{14}$ for each computation. By drawing a vertical line in Figure \ref{fig:TD_performance} (a), it becomes clear that LU computation becomes more expensive for $p$-FEM than $h$-FEM. This is expected as matrices generated from $p$-FEM are denser than those from $h$-FEM. Moreover, from the plot we see that the TIAR performance is fairly balanced for both $p$-FEM and $h$-FEM depending mostly on $N$.

Regarding accuracy, we can draw a horizontal line in any of the plots in Figure \ref{fig:TD_performance} and pick two computations that lie close to this line, meaning that both took similar computational times. Then by comparing the errors given in colors, it is apparent that $p$-FEM reachs lower errors than $h$-FEM for all computations. In Table \ref{tab:performance_TD} we list data for the numerical computation of $\lambda_{14}$ and
compare the performance for $h$-FEM and $p$-FEM by keeping accuracy fixed. The columns with $p=1,2$ are referred as $h$-FEM (low polynomial orders and several mesh refinements) and the column with $p=11$ is $p$-FEM. In the last column ($p=20$) we list data for a highly accurate $p$-FEM discretization.
The simulations suggest that for the given problems \ref{sec:SD} and \ref{sec:TD} $p$-FEM surpasses $h$-FEM in terms of performance for the proposed NEP solver. 

\begin{table}
	\centering
\begin{tabular}{|l|r|r|r|r|} 
\hline
$p$ 											& $1$ 								& $2$ 				& $11$ 				& $20$ 	\\ \hline
$|\lambda_j-\lambda^\fem_j|/|\lambda_j|$	
													& $10^{-4}$		& $10^{-4}$ 	& $10^{-4}$ 	& $1.9\times10^{-13}$  \\ \hline
$N$ 											& $1.4\times 10^6$ 		& $22657$ 		& $5697$ 			& $141121$  \\ \hline
LU time	($s$)							& $440$							& $3.23$	 		& $0.47$			&  $228.73$  \\ \hline
TIAR	time	($s$)					& $2464$						& $18.58$	 		& $4.74$			&  $149.05$  \\
\hline
\end{tabular}

	\caption{\emph{Comparison for fixed errors on $\lambda_{14}$ in problem \ref{sec:TD} with shift $\mu=20-0.3i$.}}
	\label{tab:performance_TD}
\end{table}

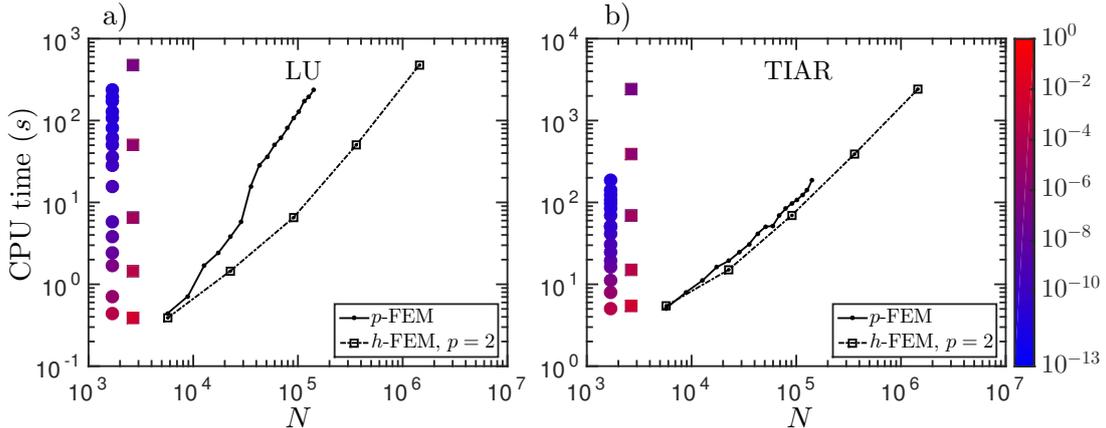
\begin{figure}
	\centering
	\begin{tikzpicture}
	
		\draw( 0.00, 6.0) node { \includegraphics[scale=0.35]{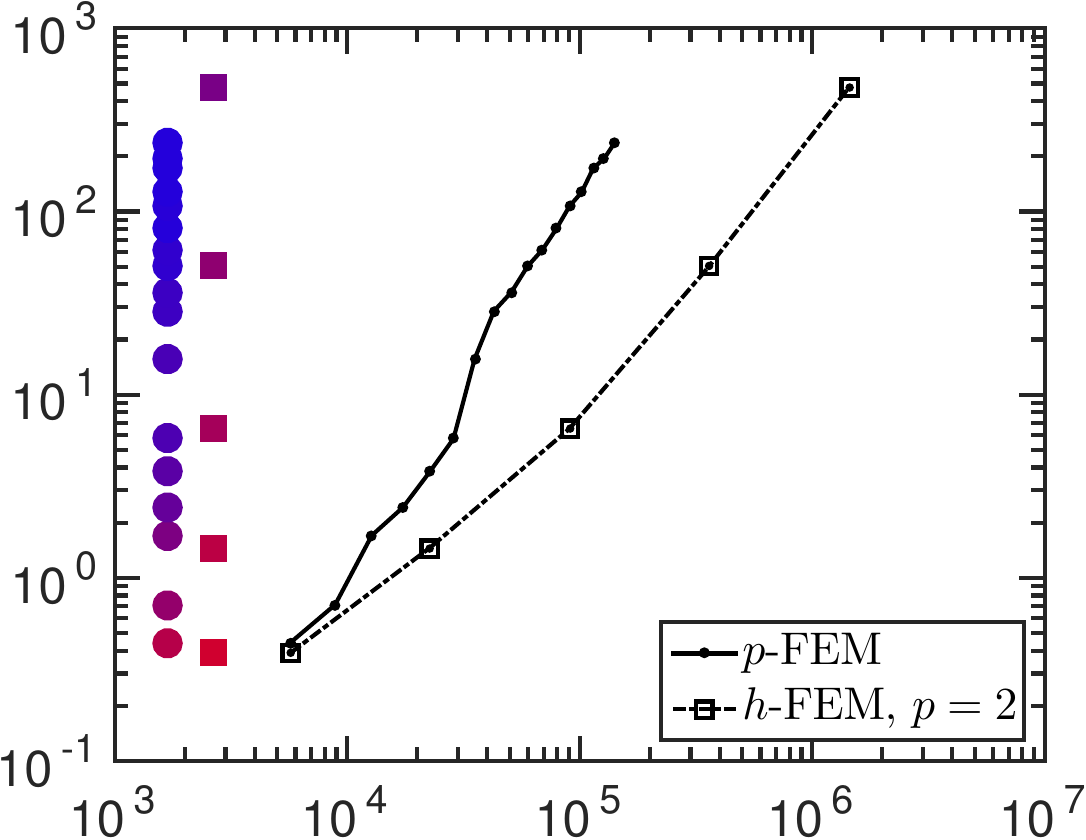}   };
		\draw( 6.60, 6.0) node { \includegraphics[scale=0.35]{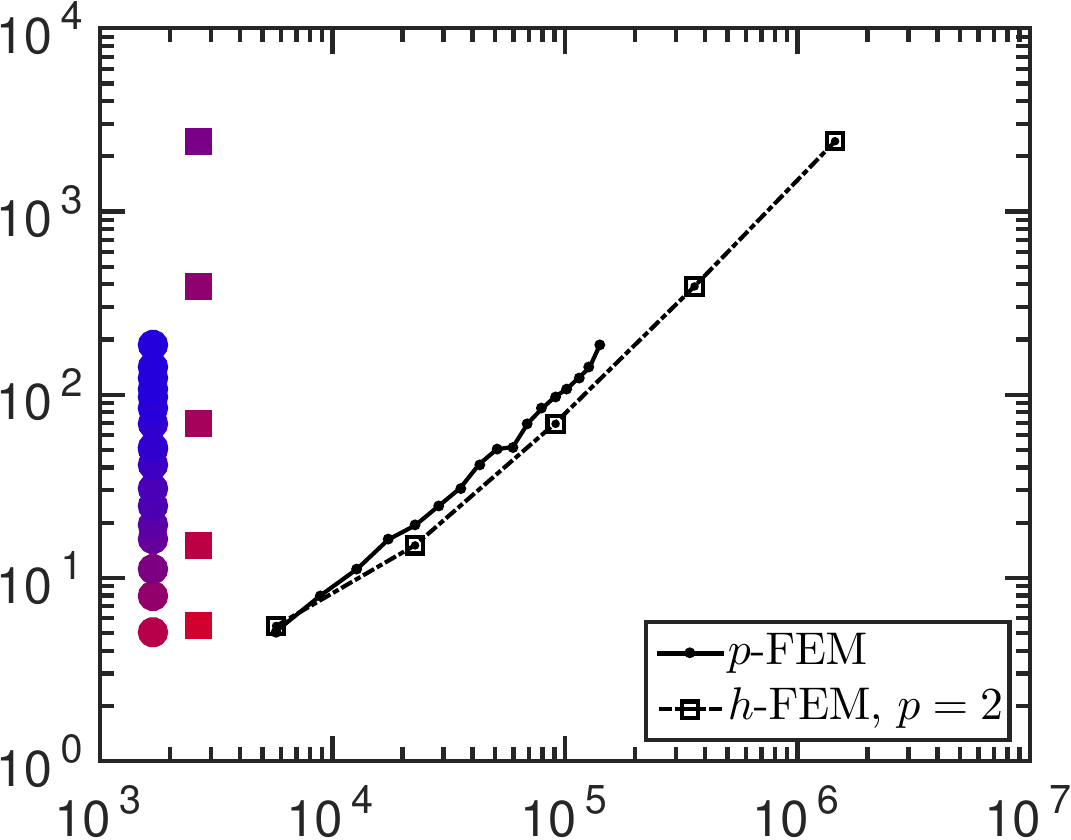} };
		\draw( 10.20, 6.2) node { \includegraphics[scale=0.35]{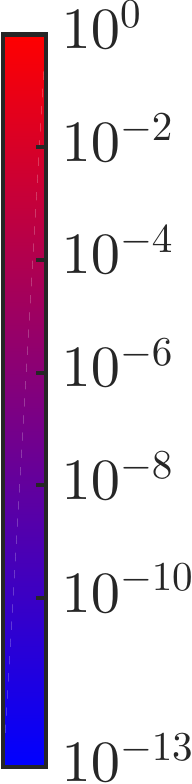} };
		
		\node[text width=5.3cm] at ( 0.30, 8.60) {a)};
		\node[text width=5.3cm] at ( 6.90, 8.60) {b)};
		
		\node[text width=0.3cm] at (  0.2, 7.9) {{\small LU}};
		\node[text width=0.3cm] at (  6.5, 7.9) {{\small TIAR}};
		
		\node[text width=3.0cm, rotate=90] at ( -3.4,  6.6) {CPU time ($s$)};
		\node[text width=0.3cm] at ( 0.2, 3.3) {$N$};
		\node[text width=0.3cm] at ( 6.8, 3.3) {$N$};
		
	\end{tikzpicture}
	\caption{\emph{a) LU and b) TIAR performance for problem \ref{sec:TD}. In colors we give relative errors for $p$-FEM (bullets) and $h$-FEM (squares) computed for $\lambda_{14}$ that correspond to the bullets/squares in black.}}
	\label{fig:TD_performance}
\end{figure}

\section{Conclusions and outlook}
We have proposed a fast and efficient method for computing resonances and resonant modes of Helmholtz problems posed as a NEP. The finite element method is used for discretizing the problem, and the resulting NEP is solved with a new specialization of the infinite Arnoldi method. In the numerical experiments we observe that the performance of the TIAR iterations scale linearly with the problem size and it is stable with respect to the number of terms in the Dirichlet-to-Neumann map. A pole cancellation technique was successfully applied to increase the radius of convergence when the shift is close to a pole.  The $h$ and the $p$ version of the finite element method were used to discretize the Fredholm operator function and we showed that the nonlinear matrix eigenvalue solver performs well in all cases. The exponential convergence of the $p$-version of FE, for problems with smooth interfaces, together with the new version of the infinite Arnoldi method is therefore an efficient tool for resonance calculations. Moreover, we expect the same performance for $hp$-FEM, since the sparsity of the matrices do not critically affect the performance of the TIAR method. \\

{\bf Acknowledgments.}
We gratefully acknowledge the support of the Swedish Research Council under Grant No. 621-2012-3863
and 621-2013-4640. J. Ara\'ujo also thanks the department of Mathematics at KTH Royal Institute of Technology very much for the kind hospitality and Giampaolo Mele for interesting discussions held during the visit.

\bibliographystyle{plain}

\end{document}